\documentclass[11pt,a4paper]{article}
\usepackage{macro}
\usepackage[a4paper,left=30mm,top=30mm,right=30mm,bottom=30mm]{geometry}

\usepackage[english]{babel}

\allowdisplaybreaks[4]

\title{Density Measures}
\author{Moritz Schönherr, Friedemann Schuricht \\ \small TU Dresden -
	Fakultät Mathematik\\ \small 01062 Dresden, Germany}
\date{}

\begin{document}

\maketitle

\vspace{5mm}

\begin{abstract}
The paper treats density measures as typical examples of finitely additive
measures in~$\R^n$. We study their structure and derive basic properties.
In addition, estimates for related integrals are provided. 
The results are applied to the precise representative of general integrable
functions and then they are specialized to functions of bounded
variation. Moreover, a new 
representation of the generalized gradients in the sense of Clarke is given
for the finite dimensional case. 
\end{abstract}

\vspace{5mm}

\section{Introduction}
\label{int}

A comprehensive theory for finitely additive measures and related integrals 
was already worked out decades ago. But, due to the lack of 
relevant examples, there are hardly any significant applications
and the related integration theory has
not been studied as extensively  
as that for $\sigma$-additive measures. 
In Schuricht \& Schönherr \cite{trace} the density $\dens_x A$ of a
set $A$ at a point $x$, which is an important tool in geometric measure
theory, was now considered as set function for fixed $x$.
It turned out that this is a typical example for a whole class 
of purely finitely additive measures and that the related integrals 
are a natural tool for the treatment of traces.
In this paper we study such measures on Borel sets 
$\dom\subset\R^n$ in greater detail and 
provide further applications. 

As a typical example of the measures under consideration, recall 
that the density $\dens_x A$ is given by 
\begin{equation} \label{int-e2}
  \dens_xA := \lim_{\de\downarrow 0} 
  \frac{\lem(A\cap B_\de(x))}{\lem(B_\de(x))}
\end{equation}
if the limit exists ($\lem$ - Lebesgue measure, $B_\de(x)$ - 
ball of radius $\de$ around $x$). 
For fixed~$x$ we consider $\dens_x$ as set
function. 
Clearly, it is disjointly additive and we have 
\begin{equation} \label{int-e4}
  \dens_x(\{x\})=0\,, \quad \dens_x(A\setminus B_\de(x))=0 \qmq{for all} 
  A\,,\:\de>0\,.
\end{equation}
Thus, if $\dens_xA>0$,   
\begin{equation*}
  0 = \lim_{\de\downarrow 0} \,\dens_x(A\setminus B_\de(x)) \ne \dens_xA =
  \dens_x(A\setminus\{x\})\,.
\end{equation*}
Therefore, $\dens_x$ cannot be countably additive but only finitely additive. 
By the Hahn-Banach theorem, $\dens_x$ can be extended (though not uniquely) 
to a purely finitely additive measure on all Borel sets of $\R^n$. 
At first glance one might
think that the measure $\dens_x$ is concentrated at $\{x\}$. But, by
\reff{int-e4} it is a measure that 
\begin{equation} \label{int-e6}
  \text{``lives'' \z in any vicinity of $x$ \z but not \z at $x$.}
\end{equation}
This behavior substantially differs from what we know about 
countably additive measures, but  
it is typical for purely finitely additive measures.
In analogy to \reff{int-e2} one can consider densities $\dens_\C$ where $x$ is
replaced by an $\cL^n$-null set $\C$ and $B_\de(x)$ by the $\de$-neighborhood of
$\C$. This leads to purely finitely additive measures   
that concentrate in any neighborhood of $\C$, but vanish on $\C$
(cf. Example~\ref{pl-s2}). 

For a finitely additive measure $\mu$ and a function $f$ the integral can be
defined as usual by approximation using simple functions,
but we have to use convergence in measure $\mu$ instead of
convergence $\mu$-almost everywhere. 
Let us briefly discuss the specifics of the integral that are induced by 
\reff{int-e6}. First we notice that $\dens_x$ 
can be considered as element of the dual space
$\cL^\infty(\R^n)^*$ of the essentially bounded functions on $\R^n$ by means
of the duality pairing 
\begin{equation*}
  \df{\dens_x}{f} = \I{\R^n}{f}{\dens_x} =: \pr f(x) \qmq{for all}
  f\in\cL^\infty(\R^n)\,.
\end{equation*}
By \reff{int-e4} it is sufficient to integrate on any ball $B_\de(x)$ ($\de>0$)
instead on $\R^n$. 
Thus $\dens_x$ has some similarity to the Dirac measure $\delta_x$
supported on $\{x\}$, since 
\begin{equation*}
  \I{B_\de(x)}{f}{\dens_x} = \I{\{x\}}{f}{\delta_x} = f(x) \qmq{for all}
  f\in C(\R^n)\,, \; \de>0
\end{equation*}
($C(\R^n)$ - continuous functions). But notice that the integral with 
$\delta_x$ exploits only the value of $f$ at $\{x\}$, 
while the integral with the density $\dens_x$ uses the
values of $f$ on $B_\de(x)\setminus\{x\}$ for any arbitrary small $\de>0$.
Roughly speaking the integral $\I{\R^n}{f}{\dens_x}$ averages $f$ 
on any small ball around $x$. This looks like a limit of averages as in
the computation of the precise representative, which is given by 
\begin{equation*}
  f^\star(x) := \lim_{\de\downarrow 0} \frac{1}{\lem(B_\de(x))}
  \I{B_\de(x)}{f}{\lem} 
\end{equation*}
if the limit exists. If we fix $f$, then $\pr f(x)$ agrees with $f^\star(x)$
for $\lem$-a.e. $x\in\R^n$. The advantage of $\pr f$ is that it exists for all
$x$ and that there is an integral calculus for its computation. 
Though $\pr f(x)$ may depend on the special extensions of $\dens_x$ from
\reff{int-e2} by the Hahn-Banach theorem for $x$ on 
an $\lem$-zero set, in practice we do not have to bother about that.
To some extend the situation is comparable to considering 
equivalence classes in $\cL^p$-spaces. It turns out that the integral 
$\I{\R^n}{f}{\dens_x}$ does not only exist for $\cL^\infty$-functions, 
but $\pr f(x)$ is also well-defined for any locally integrable function 
$f\in\cL^1_{\rm loc}(\R^n)$ and $\lem$-a.e. $x\in\R^n$.

In the present paper, we want to further investigate such measures and
associated integrals. More precisely, 
we consider positive normalized purely finitely additive
measures on Borel sets $\dom\subset\R^n$ 
that concentrate in the vicinity of an $\cL^n$-null set $\C$
contained in the closure $\ol\dom$. 
It turns out that such measures are very useful for several
applications. We call them density measures and the set of all these measures
is denoted by $\Dens_\C$. We show the existence of such
measures, verify several properties, and provide a number of applications. 

Since finitely additive measures and the corresponding integration theory are
not so widely used, in Section~\ref{pl} we give some brief introduction  
to the subject adjusted to our needs. This is a summary of the more
comprehensive exposition in \cite{trace} and also serves to fix 
notation and terminology for later use.  

In Section~\ref{dm} we first introduce density measures for $\C$ 
and show that the set $\Dens_\C$ of all these density measures  
is a convex weak$^*$ compact subset of the dual of $\cL^\infty(\dom)$.
Then we verify the existence of a large class of density measures 
including $\dens_x$ and $\dens_\C$ as discussed above. Moreover
we provide estimates for integrals related to density measures in order to 
characterize the set $\Dens_\C$. 

The convex set $\Dens_\C$ is spanned by its extreme points according to the 
Krein-Milman theorem. We show in Section~\ref{ep} that the extreme points are 
\zo measures, i.e. measures that take only the values 0 and 1. 
It turns out that such measures
concentrate in the vicinity of a point $x\in\ol\dom$. A more refined analysis
shows that they are concentrated in a certain direction,
i.e. there is a ray emanating from $x$ such that any open cone with vertex at
$x$ and containing the ray contains the full mass of the measure. 

The remaining part of the paper is devoted to applications.
Section~\ref{app-pr} studies the representation of 
the precise representative by means of integrals, first 
for functions in $\cL^1_{\rm loc}(\dom)$ and then for functions of bounded
variation. In particular we discuss the relation between $\pr f(x)$ and
$f^\star(x)$. Moreover we demonstrate how the fundamental Gauss-Green formula
can be substantially extended by using integrals related to finitely additive
measures. 
Section~\ref{app-sm} is devoted to some special problems. The
existence of numerous nontrivial measures $\mu$ with the property that
\begin{equation*}
  \I{\dom}{f}{\me}=0 \quad\qmz{for all continuous} f\in L^\infty(\dom) \,
\end{equation*}
is shown in Toland \cite[p. 28]{toland} by means of the Hahn-Banach theorem.  
We give a more explicit construction using special density
measures to get an idea of how such measures look.
In Section~\ref{cgg} we consider the generalized gradients $\partial f(x)$
in the sense of Clarke for locally Lipschitz continuous functions $f$ on
$\R^n$. First a representation of $\partial f(x)$ by means of density measures 
is provided. Then we demonstrate that this allows quite simple proofs for 
basic calculus rules. 
We are grateful to Christoph Tietz for his comments about the
representation of the generalized gradient.

Let us finally mention that some of the results are already contained in 
Schönherr \cite{phd} and Schönherr \& Schuricht \cite{dens-ar}. 

\medskip

{\it Notation.}
For a set $\als$ we write $\als^c$, $\op{int}A$, $\op{ext}A$, $\bd A$, 
and $\cl A$ for its complement, interior, exterior, boundary,
and closure, respectively, and $\chi_\als$ for its characteristic function. 
We use $\ol\R=[-\infty,+\infty]$, $\R_{\ge0}=[0,\infty)$ and  
$\widehat{\R^n}=\R^n\cup\{\infty\}$ 
is the usual one point compactification.
$\ccl A$ denotes the closure of $A\subset\R^n$ within $\widehat{\R^n}$.
We write $|a|$ for the Euclidean norm of $a\in\R^n$.
By $\al$ and $\al^\si$ we denote an algebra and an $\sigma$-algebra 
of subsets of a set $\dom$, respectively. 
$\bor{\dom}$ stands for the Borel subsets of $\dom$.
$B_\delta(x)$ is the open ball around $x$ of radius $\delta>0$,  
$B_\delta^\dom(x)$ its intersection with $\dom$, and  
$A_\delta$ is the open $\delta$-neighborhood of set $A$.
The distance of point $x$ from set $A$ is given by $\op{dist}_A x$
and $\op{conv}A$ denotes the convex hull of $A$. For a measure $\mu$ we write 
$\op{supp}\mu$ for its support and $\reme{\me}{\als}$ for its restriction to
set $A$. $\lem$ is the Lebesgue measure on $\R^n$ and
$\ham^k$ the $k$-dimensional Hausdorff measure.
$C(\dom)$ stands for the space of continuous functions, 
$L^p(\dom)$ for the space of $p$-integrable functions, and
$\cL^p(\dom)$ for the related space of equivalence classes.
$\cW^{1,p}(\dom)$ is the usual Sobolev space of $\cL^p$-functions having 
weak derivatives in $\cL^p$ and $\cB\cV(\dom)$  is the space
of functions of bounded variation. For a normed space $X$, the dual is $X^*$ 
and $\df{\cdot}{\cdot}$ is the related duality form. 
For the mean value integral we write $\mIsymb$ (where formally 
$\mI{\dom}{f}{\lem}=0$ if $\lem(\dom)=0$).

\section{Preliminaries about measures and integration}
\label{pl}

Finitely additive measures and corresponding integrals 
are the essential tools for our investigation. Since the related theory 
is not so well-known, we provide a short introduction for the convenience 
of the reader. This way we also fix notation and 
we extend some aspects for our needs. A more comprehensive
survey is contained in Schuricht \& Schönherr \cite{trace}.
A comprehensive introduction to the theory of finitely additive measures can
be found in the book of
Bhaskara Rao \& Bhaskara Rao~\cite{rao} 
(cf. also Dunford \& Schwartz \cite{dunford}). Notice that the
terminology used in the literature is not uniform. Moreover, let us emphasize
that, {\it in contrast to standard usage, 
we use the notion ``measure'' for any additive set function}. 

Let $\dom\subset\R^n$ be a Borel set. Then   
$\mu:\bord\to\ol\R$ is said to be a {\it measure} on 
$\dom$ if $\mu(\emptyset)=0$ and if $\mu$ 
is {\it additive}, i.e. for a finite index set $I\subset\N$ and 
pairwise disjoint $\als_k\in\bord$, $k\in I$, we have
\begin{equation*}
  \mu\Big(\, {\textstyle \bigcup\limits_{k\in I}} A_k\Big) = 
  \sum_{k\in I}\mu(A_k) \,.
\end{equation*}
We call $\mu$ a {\it $\sigma$-measure} if it is {\it $\sigma$-additive}, i.e.
the previous additivity is also true for $I=\N$.
For additivity to be properly defined, $\mu$ cannot attain both values
$\pm\infty$. 
$\mu$ is {\it positive} if $\mu(A)\ge 0$ and {\it finite} if $\mu(A)\in\R$ 
for all $A$. As usually the {\it positive} and {\it negative part}
$\mu^\pm:\bord\to\ol\R_{\ge 0}$ of $\mu$ are given by
\begin{equation*}
  \mu^\pm(A) := \sup\{ \pm\mu(B) \mid B\subset A\,,\: B\in\bord\}\,
\end{equation*}
and $|\mu|:\bord\to\ol\R_{\ge 0}$ with $|\mu|=\mu^++\mu^-$ is the {\it total
variation} of $\mu$. Then $\mu^\pm$, $|\mu|$ are positive measures.
We say that $\mu$ is
{\it pure} if for any $\sigma$-measure $\sigma:\bord\to\ol\R$  
\begin{equation*}
  0\le\sigma\le|\mu| \qmq{implies} \sigma =0\,.
\end{equation*}
This means that a pure measure cannot be extended to a $\sigma$-measure,
but a non $\sigma$-measure need not be pure (since it can have a
$\si$-additive part). If $\mu$ is pure, then
also $\mu^\pm$ and $|\mu|$. For a positive measure~$\mu$ the 
{\it outer measure} 
\begin{equation*}
  \mu^*(A):=\inf\limits_{\substack{B\in \bord\\A\subset B}}\mu(B) \,
\end{equation*}
is (finitely) subadditive. $A\subset\dom$ is a {\it $\mu$-null} set if
$|\mu|^*(A)=0$. By $\reme{\me}{\als}$ we denote the usual restriction of 
$\mu$ to $\als\in\bord$, which is again a measure.
Measure $\mu$ is {\it bounded} if there is some $c>0$ with 
$|\me(\als)| < c$ for all $\als$. The space
\begin{eqnarray*}
  \bab
&:=&
  \{\mu:\bord\to\R\mid \text{$\mu$ is a bounded measure}\}\,, 
\end{eqnarray*}
is a Banach space with $\|\mu\|=|\mu|(\dom)$. It turns out that any 
$\mu\in\bab$ can be uniquely decomposed into the sum of a pure and a
$\si$-additive measure.

We call $\me$ {\it absolutely continuous} with respect to  
$\lme$ ($\me\ac\lme$), if for any $\eps>0$ there is some 
$\delta>0$ such that for all $A\in\bord$
\begin{equation*}
  |\lme(A)|<\delta \qmq{implies} |\me(A)|<\eps \,.
\end{equation*}
Then also $|\me|\ac\lme$. 
We say that $\me$ is {\it weakly absolutely continuous}
with respect to $\lme$ ($\me\wac\lme$), if for all $A\in\bord$ 
\begin{equation*}
  |\lme|(A)=0 \qmq{implies} \me(A)=0\,. 
\end{equation*}
Then also $|\mu|\wac\lme$. 
For a (not necessarily bounded) measure $\lme$ on $\dom$ we define
\begin{equation*}
  \op{ba}(\dom,\bord,\la) := \{ \mu\in\bab\mid \mu\wac\lme\}\,.
\end{equation*}
Then $\op{ba}(\dom,\bord,\la)$ is a closed subspace of $\bab$ and, thus, also
a Banach space. Measures $\mu\in\bawl{\dom}$ 
that do not charge sets of $\lem$-measure zero are of special interest, 
because they are naturally suited to the integration of functions that are
only defined up to a set of $\lem$-measure zero. 

Let us now consider some example for a pure measure that is typical for the
measures we intend to study.

\begin{example}  \label{pl-s2}
Let $\dom\in\cB(\R^n)$, let $x\in\cl\dom$, and let
$\lem(\dom\cap B_\delta(x))>0$ for all $\delta>0$. 
We define the density of $A\in\cB(\dom)$ at $x$ within
$\dom$ by
\begin{equation} \label{pl-s2-1} 
\dens_x^\dom(A) := \lim_{\delta \downarrow 0} 
\frac{\lem(A\cap\dom\cap B_\delta(x))}{\lem(\dom\cap B_\delta(x))}
\end{equation}
if the limit exists. This quantity, usually considered for $\dom=\R^n$,  
is an important ingredient in geometric measure theory. Clearly,  
$\dens_x^\dom(\cdot)$ as set function is disjointly additive. Moreover
it can be extended, though not unique,
to a measure defined on $\cB(\dom)$ such that
\begin{equation*} 
  \liminf_{\delta \downarrow 0} 
  \frac{\lem(A\cap\dom\cap B_\delta(x))}{\lem(\dom\cap B_\delta(x))}
  \le \dens_x^\dom(A) \le 
  \limsup_{\delta\downarrow 0} 
  \frac{\lem(A\cap\dom\cap B_\delta(x))}{\lem(\dom\cap B_\delta(x))} \,
\end{equation*}
for all $A\in\cB(\dom)$ (cf. Corollary~\ref{dm-s2c} below with 
$C=\{x\}$). Thus $\dens_x^\dom(A)\in[0,1]$ for all~$A$.
In particular, $\dens_x^\dom(\{x\})=0$ and, clearly,  
\begin{equation} \label{pl-s2-2}
  \dens_x^\dom(B_\delta(x)\cap\dom)=1 \zmz{and} 
  \dens_x^\dom(\dom\setminus B_\delta(x))=0 \qmq{for all} \delta>0\,.
\end{equation}
If $\lem(A)=0$, then $\dens_x^\dom(A)=0$ and, thus,
\begin{equation*}
  \dens_x^\dom\in\bawl{\dom} \,.
\end{equation*}
Obviously $\dens_x^\dom$ cannot be $\sigma$-additive, 
since otherwise $\dens_x^\dom(\dom\setminus\{x\})=0$ by \reff{pl-s2-2}, 
which is a contradiction. By \reff{pl-e3} below, $\dens_x^\dom$ is pure
(cf. also \reff{pl-e2}). 

We can extend the example by choosing some 
\begin{equation} \label{pl-s2-4}
  \C\subset\ol\dom \qmq{with} \lem(\C\cap\dom)=0
\end{equation}
such that 
\begin{equation*}
  0<\lem(\dom\cap B_\delta(x)) \zmz{for all} \delta>0\,,\quad 
  \lem(\dom\cap C_{\tilde\delta})<\infty \qmq{for some} \tilde\delta>0
\end{equation*}
and by defining a density $\dens_\C^\dom$ 
at $\C$ 
as in \reff{pl-s2-1} by replacing $B_\delta(x)$ with the $\de$-neighborhood 
$\C_\delta$. Then we get again a pure measure in $\bawl{\dom}$.

Of course we can make the same construction for some $\C\subset\ol\dom$ with 
$\lem(\C\cap\dom)>0$. Then 
the restriction $\reme{\dens_\C^\dom}{(\dom\cap\cl\C)}$ 
of $\dens_\C^\dom$ on $\dom\cap\cl\C$ is a scaled version of 
$\lem$ restricted to 
$\dom\cap\cl\C$ and, hence, $\sigma$-additive. By
\begin{equation*}
  0\le\reme{\dens_\C^\dom}{(\dom\cap\cl\C)}\le\dens_\C^\dom\,,
\end{equation*}
$\dens_\C^\dom$ is not pure in this case. We readily get 
$\dens_\C^\dom(A)=0$ for all $A\subset\dom\setminus\cl\C$ and, thus,
the measure $\dens_\C^\dom$ has to be $\si$-additive at all. 
Since $\sigma$-measures are well understood, we focus on
cases where \reff{pl-s2-4} is met.
\end{example}

It turns out that the nature of pure measures as observed in the previous
example is typical for pure measures in $\bawl{\dom}$.
More precisely, in the case $\lem(\dom)<\infty$ we have that 
$\me\in\bawl{\dom}$ is pure if and only
if there exists a decreasing sequence $\{A_k\}\subset\cB(\dom)$ such that
\begin{equation} \label{pl-e3}
\lem(\als_k) \overset{k\to\infty}{\longrightarrow} 0 \quad\qmq{and}\quad
|\me|(\als_k^c) = 0 \zmz{for all} k\in\N\,.
\end{equation}
(cf. \cite[p. 244]{rao}).
This raises the question where a pure measure ``lives''.
For a $\si$-measure $\si$ the {\it support} is given by 
\begin{equation*}
  \supp\sme := \big\{x\in\dom\:\big|\: 
  |\sme|(U\cap\dom)>0 \text{ for all open }U\subset \R^n \text{ with }
  x\in U   \big\} \,.
\end{equation*}
It contains the whole mass of the measure and, thus, indicates where it
is concentrated.  
However, by \reff{pl-e3}, a pure measure
$\me\in\bawl{\dom}$ vanishes on its support. 
Thus it is inappropriate to describe where such a measure ``lives''. 
Therefore we also consider the 
{\it core} of $\me\in\bab$, a slight modification of the support, 
given by
\begin{equation}\label{pl-core}
  \cor{\me} := \big\{ x\in\R^n\:\big|\: 
  |\me|(U\cap\dom)>0 \text{ for all open }U\subset \R^n \text{ with }
  x\in U   \big\} \,.
\end{equation}
Obviously,
\begin{equation}\label{pl-e00}
  \cor{\me} \zmz{is closed in} \R^n \qmq{and}
  \cor{\me}\subset\ol\dom \,.
\end{equation}
Notice that $\cor\me$ must not be contained in $\dom$. For unbounded 
$\dom$ a nontrivial measure can also concentrate ``near $\infty$''.
In this case the core, as defined in \reff{pl-core}, is empty 
by lack of compactness of $\cl\dom$.
That we can properly treat also such cases we (tacitly) use the 
compactification
\begin{equation} \label{e-comp}
  \ccl{\R^n}:=\R^n\cup\{\infty\} \,
\end{equation}
instead of $\R^n$ in the definition \reff{pl-core}.
Here $B_\de(\infty):=\cl{B_{1/\de}(0)}^c$ are the open balls around $\infty$ 
and, in addition to the (usual) closure 
$\cl A$ of $A$ in $\R^n$, we also use
\begin{equation}\label{e-compp}
  \ccl A := \mbox{(compactified) closure of $A$ in $\ccl{\R^n}$}\,.
\end{equation}
Then we have that    
\begin{equation} \label{pl-e0}
  \cor\mu\ne\emptyset \zmz{with} \cor\mu\subset\ccl\dom
  \qmq{if} \mu\ne 0
\end{equation}
and, for all open $U\subset \ccl{\R^n}$ with $\cor{\me}\subset U$,
\begin{equation} \label{pl-e1}
  |\me|(U\cap\dom) = |\me|(\dom)\,, \quad
  |\me|(\dom\setminus U)=0
\end{equation}
(cf. \cite[p. 23]{trace}). 
Thus the full mass of the measure is contained in 
any neighborhood of the core.
For the description of that situation we call $A\in\bord$ an {\it aura} of
measure $\mu\in\bab$ if 
\begin{equation*}
  |\mu|(A^c)=0\,.
\end{equation*}
An aura is not unique, but the combination of the core and a
suitable aura allows an appropriate description where a pure measure $\me$
concentrates. 
For $\mu=\dens_x^\dom$ from Example~\ref{pl-s2} we have
\begin{equation*}
  \cor\mu=\{x\}\,,\quad \mu(\{x\})=0\,,\quad 
  \me(B_\delta(x)\cap\dom)=1
  \zmz{for all} \delta>0\,.
\end{equation*}
Thus any $A=B_\delta(x)\cap\dom$ is an aura of $\mu$. 
If $x\in\ol\dom\setminus\dom$, then $\dom\cap\cor\mu=\emptyset$ but 
$\mu(\dom)=1$. This means   
that $\mu$ concentrates near $x$ in $\dom$. 
The notion of core allows to supplement \reff{pl-e3} by the sufficient
condition that, for any $\dom\in\cB(\R^n)$, 
\begin{equation} \label{pl-e2}
  \text{$\mu\in\bawl{\dom}$ is pure} \qmq{if}
  \lem(\dom\cap\cor\mu)=0\,
\end{equation}
(cf. \cite[p. 24]{trace}).
The condition is clearly satisfied if $\dom\cap\cor\mu=\emptyset$.

Recall that $f_k$ {\it converges in measure} $\mu$ 
to $f$ (short $\f_k \to f$ i.m. $\mu$) if
\begin{equation*}
\lim_{k \to \infty} |\mu|^* \{x\in\dom \mid |f_k(x) - f(x)| > \eps \} = 0 
\qmq{for all} \eps>0\,.
\end{equation*} 
The limit is unique if we identify $f,g:\dom\to\R$ that   
{\it agree in measure} $\mu$ ($f=g$ i.m.), i.e.
\begin{equation*}
  |\mu|^*\big(\big\{x\in\dom \:\big|\: |f(x)-g(x)|>\eps \big\}\big) = 0
  \qmq{for all} \eps>0\,.
\end{equation*}
Notice that all functions $f_c(x)=c|x|$ on $\R^n$ with $c\in\R$ 
agree in measure $\dens_0^\dom$ according to Example~\ref{pl-s2}. 
We say that $f\le g$ i.m. $\mu$ if 
\begin{equation*}
  |\mu|^*\big(\big\{x\in\dom\:\big|\: g(x)-f(x)<-\eps \big\}\big)=0
  \qmq{for all} \eps>0\,.
\end{equation*}
Notice that $|\mu|^*(A)=|\mu|(A)$ for $A\in\cB(\dom)$. 

A function $f:\dom\to\R$ is $\mu$-measurable if $f_k\to f$ i.m. $\mu$ for a
sequence $f_k$ of simple Borel measurable functions. 
The integral $\I{\dom}{f}{\mu}$ of a $\mu$-measurable function $f$ is defined
as usually by approximation with simple functions 
if one replaces convergence ``$\mu$-a.e.'' by convergence ``i.m. $\mu$''.
We have that $\mu$-integrability agrees
with $|\mu|$-integrability and $f$ is integrable if and
only if $|f|$ is integrable. 
This integral shares many properties with the usual one 
(e.g. standard calculus rules, dominated convergence) if we replace
``$\mu$-a.e.'' by ``i.m. $\mu$''. For example the integral vanishes if 
$f=0$ i.m. $\mu$.
For $\sigma$-measures, the integral agrees with that in the usual  
``$\si$-theory''.  

For a $\sigma$-measure it is sufficient to integrate on its support.
However, for a pure measure we have to integrate on an aura (which is
not unique), whereas integrating on the core is insufficient. 
We use the notation
\begin{equation*}
  \sI{C}{f}{\mu} := \I{\dom\cap U}{f}{\mu} \qmq{for any $\C$ and open $U$ with}
  \cor\mu\subset C \subset U 
\end{equation*}
to give a more precise indication of the domain of integration
(cf. \reff{pl-e1} for justification). 

For a measure $\mu$ on $\dom\subset\cB(\R^n)$
and for $1\le p\le\infty$ we define
as usually
\begin{equation*}
  L^p(\dom,\cB(\dom),\me) := \big\{ f:\dom\to\R\:\big|\: 
  \|f\|_p<\infty  \big\}
\end{equation*}
where
\begin{equation}\label{pl-e10}
  \|f\|_p := \Big( \I{\dom}{|f|^p}{|\mu|}\Big)^{\frac1p} 
  \quad (1\le p<\infty)\,,
  \quad
  \|f\|_\infty := \essup{\dom} |f|\,.
\end{equation}
The corresponding spaces of equivalence classes with respect to
equality i.m., denoted by
\begin{equation*}
  \cL^p(\dom,\cB(\dom),\me) \,,
\end{equation*}
are normed spaces, but they are not complete in general. For a
$\sigma$-measure $\mu$ the spaces agree with
the usual ones. We briefly write
\begin{equation*}
  L^p(\dom):= L^p(\dom,\bor{\dom},\lem)  \,, \quad
  \cL^p(\dom):= \cL^p(\dom,\bor{\dom},\lem) \,.
\end{equation*}
The dual of $\cL^\infty(\dom,\cB(\dom),\lem)$ is of special relevance for our
investigations and we have
\begin{equation} \label{pl-e6}
   \cL^\infty(\dom,\cB(\dom),\lem)^* =  \bawl{\dom}  
\end{equation}
if we identify $f^* \in \cL^\infty(\dom,\cB(\dom),\lem)^*$ with 
$\me\in\bawl{\dom}$ such that 
\begin{equation*}
  \df{f^*}{f} = \I{\dom}{f}{\mu} \qmq{for all}
  f \in \cL^\infty(\dom,\cB(\dom),\lem)\,
\end{equation*}
where $\|f^*\| = \|\mu\| = |\mu|(\dom)$. 
For $\mu\in\bawl{\dom}$ we thus have that all $f\in\cL^\infty(\dom)$ are
$\mu$-integrable, i.e. 
\begin{equation*}
  \cL^\infty(\dom) \subset \cL^1(\dom,\cB(\dom),\mu)\,.
\end{equation*}
It turns out that also unbounded functions can be
$\mu$-integrable. 
For measures $\dens_x^\dom$ as in Example~\ref{pl-s2},
integrability conditions for $f\in\cL^1_{\rm loc}(\dom)$ 
are formulated in Theorem~\ref{app-s1} below
(cf. also \cite{trace} for more general results).

\section{Density Measures}
\label{dm}

The density of a set $A\subset\R^n$ at a point $x$ as used in geometric
measure theory and, more generally, pure measures that concentrate near 
a set $\C$ as in Example \ref{pl-s2} turn out to be relevant in several
applications. We consider measures of that kind as density measures
and investigate their properties. 
Our definition of such measures requires them to be 
normalized. 
More generally, 
any nontrivial bounded positive measure, whose core has Lebesgue
measure zero, could be seen as a density measure. 

Let $\dom\in\bor{\R^n}$ and let $\C\subset\cl\dom$ be such that
\begin{equation} \label{dm-e0a}
  \C  \zmz{is closed (in $\R^n$)\,,}\quad   \lem(\C\cap\dom) = 0 
  \qmq{and} \C\ne\emptyset\,.
\end{equation}
A measure $\me\in\bawl{\dom}$ is called 
{\it density measure} (short {\it density}) {\it at} $\C$ (on $\dom$) if 
\begin{equation} \label{dm-e0}
  \me\ge 0 \qmq{and}  \me(\Cd\cap\dom) = \me(\dom) = 1
  \qmq{for all} \delta>0\,.
\end{equation}
The set of all density measures at $\C$ (on $\dom$) is denoted by
\begin{equation*}
\Dens_\C \qquad (\text{briefly also } \Dens_x \text{ if } C={\{x\}})\,.		
\end{equation*}
Notice that $\Dens_\C=\emptyset$ if $\lem(\dom \cap \dnhd{\C}{\delta}) = 0$ 
for some $\delta >0$. Thus $\Dens_\C\ne\emptyset$ implies that
\begin{equation} \label{dm-e0b}
  \lem(\Cd\cap\dom)>0 \zmz{for all} \delta>0\,.
\end{equation}
We say that $\C\subset\cl\dom$ is a {\it density set} (of $\dom$) 
if it satisfies \reff{dm-e0a} and \reff{dm-e0b}. 
Corollary~\ref{dm-s2c} below implies that $\Dens_\C\ne\emptyset$ 
for all bounded density sets $\C\subset\cl\dom$. 
We now show that density measures are pure and have their core in $\ccl\C$
(cf. \reff{e-compp}).
\begin{theorem} \label{dm-s1}   
Let $\dom \in \bor{\R^n}$, let $\C\subset\cl\dom$ be a density set,
and let $\me\in\Dens_\C$. Then
\begin{equation} \label{dm-s1-1}
  \cor\mu\ne\emptyset\,, \quad \cor{\me} \subset \ccl\C\,, \quad
  \|\me\|=|\me|(\dom)=1\,,
\end{equation}
and $\me$ is pure. Moreover
\begin{equation} \label{dm-e1}
  \I{\dom}{f}{\me} = \sI{\ccl\C}{f}{\me} = \lim_{\delta \downarrow 0}
  \I{\Cd\cap\dom}{f}{\me} = \lim_{\delta\downarrow 0} 
  \mI{\Cd\cap\dom}{f}{\me} \,
\end{equation}
for all $f\in\cL^\infty(\dom)$. 
\end{theorem}
\noi
Notice that $\ccl\C=\C$ if $\C$ is bounded and $\ccl\C=\C\cup\{\infty\}$ if
$\C$ is unbounded.

\begin{proof}   
We have $\cor\mu\ne\emptyset$ by \reff{pl-e0}.
Let $x\in\R^n\setminus\C$. Then 
$\delta:=\frac{1}{2}\distf{\C}(x)>0$ and 
\begin{equation*}
\me(\dom\cap B_{\tilde\delta}(x)) \le \me(\dom\setminus\C_\delta) = 0 
\qmq{for all} 0<\tilde\delta<\delta\,.
\end{equation*}
Hence $x\notin\cor{\me}$ and, thus, $\cor{\me}\subset\ccl\C$.
Moreover
\begin{equation*}
\lem(\cor{\me}\cap\dom) \le \lem(\ccl\C\cap\dom) =
\lem(\C\cap\dom) = 0 \,.
\end{equation*}
Therefore $\me$ is pure by \reff{pl-e2}. Since $\me\ge 0$, we have 
$\|\me\|=\me(\dom)=1$.
\reff{dm-e1} readily follows from \reff{dm-e0} and \reff{dm-s1-1}.
\end{proof}

Since $\cor\me\subset\ccl\dom$ for a density measure at $\C$, one
could consider density measures for closed $\C\subset\ccl\dom$ in
\reff{dm-e0a}. 
However, the case presented here appears to be appropriate for our 
applications. 
Now we show that the set $\Dens_\C$ of all density measures is
convex and weak* compact (as set in the dual of $\cL^\infty(\dom)$).
\begin{theorem} \label{dm-s1a} 
Let $\dom\in \bor{\R^n}$ and let $\C\subset\cl\dom$ be a density set.
Then $\Dens_\C$ is a convex weak* compact subset of $\bawl{\dom}$.
\end{theorem}

\begin{proof}    
Let $\Dens_\C\ne\emptyset$ (otherwise the statement is trivial), let
$\me_1,\me_2\in\Dens_\C$, and let $\alpha\in [0,1]$. 
By \reff{dm-e0},
\begin{equation*}
  \big(\alpha\me_1+(1-\alpha)\me_2\big)(B) =
  \alpha\me_1(B)+(1-\alpha)\me_2(B) \ge 0 \qmq{for all} B\in\bor\dom\,.
\end{equation*}
Hence $\alpha\me_1+(1-\alpha)\me_2\ge 0$. Moreover, for all $\delta>0$,
\begin{eqnarray*}
\big(\alpha\me_1+(1-\alpha)\me_2\big)(\Cd\cap\dom)
&=&
\alpha \me_1(\Cd\cap\dom) + (1-\alpha)\me_2(\Cd\cap\dom) \\
&=&
\alpha \me_1(\dom) + (1-\alpha)\me_2(\dom) \\
&=& 
1 \: = \:
\big(\alpha \me_1 + (1-\alpha)\me_2\big)(\dom) \,.
\end{eqnarray*}
Thus $\Dens_\C$ is convex. 

For any fixed $\lme \in \bawl{\dom} \setminus \Dens_\C$ we now construct some   
weak$^*$ open neighborhood $V(\lme)$ of $\lme$ with
\begin{equation} \label{dm-s1a-5}
  V(\lme)\cap\Dens_\C = \emptyset\,.
\end{equation}
For that we distinguish several cases. First let 
$\lme\not\ge 0$. Then there is some $B\in\bor{\R^n}$ with
$\lme(B)<0$. Hence $\gamma:=\df{\lme}{\chi_B}=\I{\dom}{\chi_B}{\lme}<0$ and
\begin{equation*}
   V(\lme) := \big\{\me\in\bawl{\dom} \:\big|\:
               |\df{\me-\lme}{\chi_B}|<\tfrac{\gamma}{2} \big\}
\end{equation*}
is a weak* open neighborhood of $\lme$ with
\reff{dm-s1a-5}.
Let now $\lme\ge 0$ and $\lme(\Cd\cap\dom)\ne 1$ for some
$\delta>0$. Since 
$\lme(\Cd\cap\dom)=\I{\Cd\cap\dom}{}{\lme}=\df{\lme}{\chi_{\Cd\cap\dom}}$,
\begin{equation*}
   V(\lme) := \big\{\me\in\bawl{\dom} \:\big|\:
               |\df{\me-\lme}{\chi_{\Cd\cap\dom}}|<|\lme(\Cd\cap\dom)-1| \big\}
\end{equation*}
is a weak* open neighborhood of $\lme$ with \reff{dm-s1a-5}. In the case 
$\lme(\dom)\ne 1$ we argue analogously. This verifies \reff{dm-s1a-5}
for all considered $\lme$. Consequently, the complement of 
$\Dens_\C$ is weak* open. Thus $\Dens_\C$ is a weak* closed subset
of the closed unit ball, that is weak* compact by the Alaoglu theorem.
Hence $\Dens_\C$ itself is weak* compact (cf. also \cite[p. 55]{clarke-fa}). 
\end{proof}

Let us now extend the idea of Example~\ref{pl-s2} by replacing the Lebesgue
measure $\lem$ with some $\lme\in\bawl{\dom}$ satisfying
\begin{equation} \label{dm-e3} 
  \lme\ge 0\,, \quad \lme(\Cd\cap\dom) > 0 \qmq{for all} \delta>0\,.
\end{equation}
We consider densities with respect to $\lme$ instead of $\lem$. 

\begin{proposition} \label{dm-s2e}  
Let $\dom\in\bor{\R^n}$, let $\C\subset \cl{\dom}$ be a density set, and let
$\lme\in\bawl{\dom}$ be such that \reff{dm-e3} holds.
If $\me\in\bawl{\dom}$ satisfies
\begin{equation} \label{dm-s2e-2}
   \liminf_{\delta \downarrow 0} 
  \frac{\lme(A\cap\dom\cap\C_\delta)}{\lme(\dom\cap\C_\delta)}
  \le \me(A) \le 
  \limsup_{\delta\downarrow 0} 
  \frac{\lme(A\cap\dom\cap\C_\delta)}{\lme(\dom\cap\C_\delta)} \qmq{for all}
  A\in\cB(\dom)\,,
\end{equation}
then $\mu\in\Dens_\C$ and $\me\wac\lme$.
\end{proposition} 
\noi
We call $\me\in\bawl{\dom}$ satisfying \reff{dm-s2e-2} also  
{\it $\lme$-density (on $\dom$) at $\C$} and we denote 
the set of all $\lme$-densities at $\C$ by $\lDens_\C$.
Obviously, $\dens_x^\dom$ according
to Example~\ref{pl-s2} belongs to $\LDens_x$. 

\begin{proof}
Clearly $\me\ge 0$ and $\me(\dom)=\me(\C_\delta\cap\dom)=1$. Thus 
$\mu\in\Dens_\C$ and \reff{dm-s2e-2} readily implies $\me\wac\lme$.
\end{proof}

\begin{corollary} \label{dm-s2g}
Let $\dom\in\bor{\R^n}$, let $\C\subset \cl{\dom}$ be a density set, and let 
$\lme\in\bawl{\dom}$ satisfy \reff{dm-e3}. 
If $\me\in\bawl{\dom}$ is a $\lme$-density at $\C$, i.e. it satisfies 
\reff{dm-s2e-2}, then 
\begin{equation}  \label{dm-s2-1}
  \liminf_{\delta \downarrow 0}
  \mI{\Cd\cap\dom}{f}{\lme} \le \sI{\ccl\C}{f}{\me} 
\le \limsup_{\delta\downarrow 0} 
  \mI{\Cd\cap\dom}{f}{\lme} 
 \qmq{for all} f\in\cL^\infty(\dom)\,.
\end{equation}
\end{corollary}
\noi
Notice that \reff{dm-s2e-2} follows from \reff{dm-s2-1} with $f=\chi_A$.

\begin{proof}
\reff{dm-s2e-2} directly gives \reff{dm-s2-1} for $f=\chi_A$ and, by
multiplying with $-1$, also for $f=-\chi_A$. Then, by
linearity, \reff{dm-s2-1} is valid for all step functions 
(having a range with finitely many values). Since any 
$f\in\cL^\infty(\dom)$ can be uniformly approximated by step functions, the
assertion follows. 
\end{proof}

Let us now consider the existence of density measures. 
It turns out that every measure $\lme\in\bawl{\dom}$ satisfying
\reff{dm-e3} gives a density measure. This shows that a large variety of density
measures exist. 

\begin{proposition} \label{dm-s2}     
Let $\dom\in\bor{\R^n}$, let $\C\subset \cl{\dom}$ be a density set, and let
$\lme\in\bawl{\dom}$ satisfy \reff{dm-e3}. 
Then there exists a $\lme$-density $\me\in\Dens_\C$. 
\end{proposition}
\noi
Let us discuss the result before giving the proof after
Remark~\ref{dm-s2f}. By Corollary~\ref{dm-s2g} we have that $\me$
satisfies \reff{dm-s2-1}. The existence of $\mu$ 
relies on the Hahn-Banach Theorem, where it is not 
uniquely determined by \reff{dm-s2-1}. This non-uniqueness can also be seen
directly if we choose a suitable disjoint decomposition 
$\dom=\dom_1\cup\dom_2$ and apply the proposition to $\dom_1$ and $\dom_2$
separately. Then the zero extensions of the measures obtained differ. 

We can apply Proposition~\ref{dm-s2} to $\lme=\reme{\lem}{\dom}$ if 
$\lem(\dom)<\infty$. More generally, we can take  
the zero extension of $\reme{\lem}{(\dom\cap\C_{\tilde\delta})}$ as $\lme$
if it is bounded for some $\tilde\delta>0$. 

\begin{corollary} \label{dm-s2c}
Let $\dom\in\bor{\R^n}$, let $\C\subset \cl{\dom}$ be a density set,
and let
\begin{equation}\label{dm-s2c-0}
  \lem(\dom\cap\C_{\tilde\delta}) <\infty
  \qmq{for some} \tilde\delta>0\,.
\end{equation}
Then there exists an $\lem$-density $\dens_\C\in\Dens_\C$. 
\end{corollary}
\noi
Notice that, as in the previous proposition, the verified density 
$\dens_\C$ is not unique.
For $\C=\{x\}$, the corollary verifies the existence 
of a density measure $\dens_x$ as considered in Example~\ref{pl-s2}. 
We also obtain directly the next remark.

\begin{remark} \label{dm-s2f} 
We have that $\Dens_\C\ne\emptyset$ for bounded density sets $\C$.
\end{remark}

\begin{proof+}{ of Proposition~\ref{dm-s2}}    
Let $\lme \in \bawl{\dom}$ satisfy \reff{dm-e3}.
For $X := \libld{\dom}$ we consider $p:X\to\R$ with 
\begin{equation*}
  p(f) = \limsup_{\delta \downarrow 0 }
         \mI{\dnhd{\C}{\delta} \cap \dom}{f}{\lme}
\end{equation*}
which is well defined, positively homogeneous and subadditive.
Then  
\begin{equation*}
X_0 := \Big\{f \in X \:\Big|\: 
\lim_{\delta\downarrow 0} \mI{\dnhd{\C}{\delta} \cap \dom}{f}{\lme}
\;\text{ exists}\Big\} 
\end{equation*}
is a linear subspace of $X$. Clearly, $f_0^*:X_0\to\R$ with
\begin{equation*}
  f_0^*(f) = \lim_{\delta\downarrow 0}
  \mI{\dnhd{\C}{\delta} \cap \dom}{f}{\lme}
\end{equation*}
is linear and bounded from above by $p$.  
The Hahn-Banach Theorem
yields a linear extension $f^*$ of $f_0^*$ onto $X$ such that
\begin{equation*}
\df{f^*}{f} \leq p(f) \leq
\|f\|_\infty  \limsup_{\delta \downarrow 0 }
         \mI{\dnhd{\C}{\delta} \cap \dom}{}{\lme} = \|f\|_\infty
\qmq{for all} f\in X\,.
\end{equation*}
Hence, $f^*\in \libld{\dom}^*$. By \reff{pl-e6} there is some
$\me \in \bawl{\dom}$ with 
\begin{equation}\label{dm-s2-5}
  \df{f^*}{f} = \I{\dom}{f}{\me} \z
  \Big( \le \limsup\limits_{\delta \downarrow 0}
	\mI{\dnhd{\C}{\delta} \cap \dom}{f}{\lme} \:  \Big)
  \qmq{for all} f \in \libld{\dom}  \,.
\end{equation}
For every $f \in \libld{\dom}$ we have 
\begin{equation*}
	\I{\dom}{-\fun}{\me} \leq p(-f) = \limsup\limits_{\delta \downarrow 0}
	\mI{\dnhd{\C}{\delta} \cap \dom}{- \fun}{\lme} 
\end{equation*}
and, consequently,
\begin{equation}\label{dm-s2-6}
\liminf\limits_{\delta \downarrow 0}
\mI{\dnhd{\C}{\delta} \cap \dom}{\fun}{\lme} \leq \I{\dom}{f}{\me} 	\,.
\end{equation}
Thus, for every $\bals\in \bor{\dom}$,
\begin{equation*}
0 \leq \liminf\limits_{\delta \downarrow 0}
\mI{\dnhd{\C}{\delta} \cap \dom}{\chi_{\bals}}{\lme} \leq \me(\bals) \,.
\end{equation*}
Hence, $\me \geq 0$. Moreover  
\begin{equation*}
	1 = \liminf\limits_{\delta \downarrow 0}
	\mI{\dnhd{\C}{\delta} \cap \dom}{\chi_{\dom}}{\lme} \leq
	\me(\dom) \leq \limsup\limits_{\delta
	\downarrow 0}
	\mI{\dnhd{\C}{\delta} \cap \dom}{\chi_{\dom}}{\lme} = 1 
\end{equation*}	
and, for any $\tilde{\delta} > 0$,
\begin{equation*}
	1 = \liminf\limits_{\delta \downarrow 0}
	\mI{\dnhd{\C}{\delta} \cap \dom}{\chi_{\dnhd{\C}{\tilde{\delta}}
	\cap \dom}}{\lme} \leq
	\me(\dnhd{\C}{\tilde{\delta}} \cap \dom) \leq \limsup\limits_{\delta
	\downarrow 0}
	\mI{\dnhd{\C}{\delta} \cap \dom}{\chi_{\dnhd{\C}{\tilde{\delta}}
	\cap \dom}}{\lme} = 1 \,.
\end{equation*}
Thus, $\me$ is a density measure at $\C$ on $\dom$. 
From \reff{dm-s2-5} and \reff{dm-s2-6} we obtain \reff{dm-s2-1} and, with
$f=\chi_A$, also \reff{dm-s2e-2}. Thus $\me$ is a $\lme$-density at $\C$
and $\mu\in\Dens_\C$ by Proposition~\ref{dm-s2e}.
\end{proof+}

The special case where $\C=\{x\}$ for some $x\in\cl\dom$ plays an
important role in our further investigations. For $\Dens_x\ne\emptyset$
we need that $x$ belongs to  
\begin{equation} \label{cdom}
  \cdom := \{y\in\cl\dom\mid \lem(\dom\cap B_\delta(y))>0
              \zmz{for all} \delta>0\}\, \qquad
\end{equation}
(cf. \reff{dm-e0b}).
The next example indicates how Corollary~\ref{dm-s2c} provides a large
variety of density measures in $\Dens_\C$.

\begin{example} \label{dm-s2b} 
Let $\dom\in\bor{\R^n}$, let $\C$ satisfy \reff{dm-e0a},  
and let $\dom'\in\cB(\dom)$ have a 
cusp at some $x\in\C\cap\cdom$ such that
\begin{equation*}
  \lem(B_\delta(x) \cap \dom') > 0 \zmz{for all} \delta>0\,.
\end{equation*}
(cf. Figure~\ref{fig:cusp_dens}).
By Corollary~\ref{dm-s2c} with $\dom'$ instead of $\dom$, we 
obtain an $\lem$-density $\dens_x^{\dom'}$ on $\dom'$ at $x$, that we identify
with its zero extension on $\dom$. 
Then $\dens_x^{\dom'}$ is a density measure on $\dom$ at
$x$, i.e. we have $\dens_x^{\dom'}\in\Dens_x$ with aura $\dom'$. 
Clearly, the selection of different sets $\dom'$ that are pairwise disjoint near
$x$ gives different density measures at $x$. In addition we can vary 
$x\in\C$. This way we obtain a huge amount of measures in $\Dens_\C$. 

\begin{figure}[H]
	\centering
	\begin{tikzpicture}[scale=1.1]
		\draw (2,0) to[out=180,in=0] (0,1);
		\draw (0,1) to[out=180,in=90] (-1,0);
		\draw (-1,0) to[out=-90,in=180] (0,-1);
		\draw (0,-1) to[out=0,in=180] (2,0);
		\draw (0,0) node {$\dom'$};
                \draw (1,1) node {$\dom$};
	    
                \draw[line width=2pt] (2,0) circle [radius = 0.03];  
		\draw (2.2,0) node {$x$};
		\draw[dashed] (2,0) circle [radius = 1.0];
		\draw[dashed] (2,0) circle [radius = 0.5];
\end{tikzpicture}
\caption{Existence of a density measure at a cusp.}\label{fig:cusp_dens} 
\end{figure}
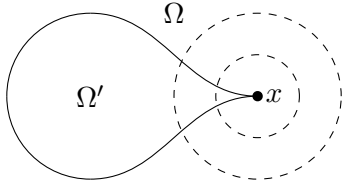
\end{example}

Let us now provide a characterization of the density measures in $\Dens_\C$.
Here we in particular show that $\me\in\bawl{\dom}$ already satisfies
\reff{dm-s2-1} if one of the two inequalities is satisfied.  

\begin{theorem}  \label{dm-s3}  
Let $\dom\in\bor{\R^n}$, let $\C\subset\cl{\dom}$ be a density set, and let 
$\me \in \bawl{\dom}$. 
Then $\me\in\Dens_\C$ if and only if there is some 
$\lme\in\bawl{\dom}$ satisfying \reff{dm-e3}  
such that
\begin{equation} \label{dm-s3-1}
\sI{\ccl\C}{f}{\me} \le \limsup_{\delta\downarrow 0} 
\mI{\Cd\cap\dom}{f}{\lme} \qmq{for all}
f\in\cL^\infty(\dom)\,.
\end{equation}
In this case \reff{dm-s3-1} implies
\begin{equation} \label{dm-s3-2}
\liminf_{\delta\downarrow 0} \mI{\Cd\cap\dom}{f}{\lme}
\le \sI{\ccl\C}{f}{\me} \qmq{for all}
f\in\cL^\infty(\dom)\,.
\end{equation}
\end{theorem}
\noi
The statement remains true if we interchange \reff{dm-s3-1} and \reff{dm-s3-2}.
If $\me\in\Dens_\C$ is given, then, by \reff{dm-e1}, we can choose 
$\lme=\me$ for the implication in the proposition. 

\begin{proof}   
Let $\me \in \bawl{\dom}$, let $\lme \in \bawl{\dom}$ satisfy \reff{dm-e3},
and let \reff{dm-s3-1} be satisfied. 
Using $\I{\dom}{f}{\me}=\sI{\ccl\C}{f}{\me}$, we can argue as in the
proof of Proposition~\ref{dm-s2} after \reff{dm-s2-5} to get 
\reff{dm-s3-2} and, with $f=\chi_A$, also \reff{dm-s2e-2}.
Hence $\me\in\Dens_\C$ by Proposition~\ref{dm-s2e}.

Now let $\me\in \Dens_\C$. Then $\lme := \me$ satisfies \reff{dm-e3}.
From \reff{dm-e1} we get 
\begin{equation*}
\sI{\ccl\C}{\fun}{\me} = 
\lim_{\delta \downarrow 0} \mI{\dnhd{\C}{\delta}\cap \dom}{\fun}{\me} \leq 
\limsup_{\delta\downarrow 0} \mI{\dnhd{\C}{\delta}\cap \dom}{\fun}{\lme} 
\qmq{for all} f\in\libld{\dom}
\end{equation*}
which completes the proof. 
\end{proof}

After that characterization of density measures we are looking for estimates
of the related integrals. 
For each measure $\me\in \bawl{\dom}$ one has the standard estimate
\begin{equation*}
  \I{B}{f}{\me} \le |\me|(B) \: \essup{B}{f} \qmq{for}
  f\in\cL^\infty(\dom)\,, \; B\in\cB(\dom)\,.
\end{equation*}
If $\me$ is a density measure at $\C$, then the estimate with 
any neighborhood $\C_\delta$ of $\C$ instead of $B$ is of special relevance. 
But we are interested in a more refined version. 
If $f\in\cL^\infty(\dom)$ and if $\C\subset\cl\dom$ is a density set,
we have that  
\begin{equation*}
  \essup{\Cd\cap\dom}{f} \qmq{and} \essinf{\Cd\cap\dom}{f}
\end{equation*}
are both monotone in $\delta>0$. Thus 
the {\it essential supremum} and {\it essential infimum} of $f$ {\it near}
$\C$ given by 
\begin{eqnarray*}
  \sC f 
&:=&
  \lim_{\delta\downarrow 0}\, \essup{\Cd\cap\dom}{f} =
  \inf_{\delta>0}\, \essup{\Cd\cap\dom}{f}  \qmq{and}  \\
  \iC f 
&:=&
  \lim_{\delta\downarrow 0}\, \essinf{\Cd\cap\dom}{f} =
  \sup_{\delta>0}\, \essinf{\Cd\cap\dom}{f}   \,,
\end{eqnarray*}
respectively, are well-defined for $f\in\cL^\infty(\dom)$. 
For $\C=\{x\}$ the quantities 
\begin{equation*}
  \ix f \qmq{and} \sx f
\end{equation*}
could also be considered 
as essential limit inferior or superior of $f$ as $y\to x$.
Now we can give a refined estimate for integrals. 

\begin{proposition} \label{dm-s4}   \label{dm-s4} 
Let $\dom\in \bor{\R^n}$, let $\C\subset\cl{\dom}$ be a density set,
and let $\me\in\Dens_\C$
be a density measure at $\C$. Then 
\begin{equation} \label{dm-s4-1}
\iC f \, \le \, \sI{\ccl\C}{f}{\me} \, \le \, \sC f 
\qmq{for all} \f\in\cL^\infty(\dom)     \,.
\end{equation}
\end{proposition}

\begin{proof}   
From \reff{dm-e1} we readily get
\begin{equation*}
  \sI{\ccl\C}{f}{\me} = \lim_{\delta\downarrow 0}\I{\Cd\cap\dom}{f}{\me} \le
  \limsup_{\delta\downarrow 0}
  \I{\Cd\cap\dom}{\essup{\Cd\cap\dom} f}{\me} 
  = \lim_{\delta\downarrow 0}\, \essup{\Cd\cap\dom} f \,.
\end{equation*}
The other inequality follows analogously. 
\end{proof}

It turns out that the inequalities in \reff{dm-s4-1} are sharp for bounded
density sets $\C$.   

\begin{theorem} \label{dm-s5}    
Let $\dom\in\bor{\R^n}$ and let $\C\subset\cl{\dom}$ be a bounded density set. 
Then we have for 
all $f\in\cL^\infty(\dom)$ that
\begin{equation}\label{dm-s5-1}
\inf_{\me \in \Dens_\C} \sI{\ccl\C}{f}{\me} = \iC f \quad\qmq{and}\quad
\sup_{\me \in \Dens_\C} \sI{\ccl\C}{f}{\me} = \sC f \,.
\end{equation}
\end{theorem}

\begin{proof}   
For $f\in\libld{\dom}$ and $\eps>0$ we set
\begin{equation*}
M_\eps := \{x\in\dom \setpipe f(x) \ge 
\lim_{\delta \downarrow0}\essup{\dnhd{\C}{\delta} \cap \dom}{\fun} - \eps\}\,.
\end{equation*}
Clearly $\C_\de\cap M_\ep$ is bounded and 
\begin{equation*}
\lem(\C_{\delta}\cap M_\eps) > 0  \zmz{for all} \delta>0 \,.
\end{equation*}
By Corollary~\ref{dm-s2c} with $M_\eps$ instead of $\dom$, 
there is some $\lem$-density $\me_\eps$ on $M_\eps$ at $\C$, that we identify
with its zero extension on $\dom$. Thus we obtain  
a density measure $\mu_\eps\in\Dens_\C$ with aura  
$M_\eps$. Then, by \reff{dm-e1},
\begin{equation*}
\sI{\ccl\C}{f}{\me_\eps} = 
\lim_{\delta\downarrow 0} \mI{\C_\delta\cap\dom}{f}{\me_\eps} =
\lim_{\delta\downarrow 0} \mI{\C_\delta\cap M_\eps}{f}{\me_\eps}
\geq \lim \limits_{\delta\downarrow 0}
\essup{\dnhd{\C}{\delta} \cap \dom}{\fun} - \eps \,.
\end{equation*}
Consequently,
\begin{equation*}
\sup_{\me \in \Dens_\C} \sI{\ccl\C}{f}{\me} \ge
\sup_{\eps > 0} \sI{\ccl\C}{f}{\me_\eps} \ge
\lim \limits_{\delta\downarrow 0}
\essup{\dnhd{\C}{\delta} \cap \dom}{\fun} \,.
\end{equation*}
From Proposition \ref{dm-s4} we get
\begin{equation*}
\sup_{\me \in \Dens_\C} \sI{\C}{f}{\me}
\leq \lim \limits_{\delta\downarrow 0}
\essup{\dnhd{\C}{\delta} \cap \dom}{\fun} 
\end{equation*}
which implies the second equality of \reff{dm-s5-1}.
The other one 
follows analogously.
\end{proof}

In order to characterize the action of $\Dens_\C$ on a 
fixed essentially bounded function, we define 
\begin{equation*}
  \df{\Dens_\C}{f} := \big\{ \df{\me}{f} \:\big|\: \me\in\Dens_\C \big\}\,.
\end{equation*}

\begin{proposition} \label{dm-s7} 
Let $\dom\in \bor{\R^n}$, let $\C\subset\cl{\dom}$ be a density set, 
let $\Dens_\C\ne\emptyset$, and assume that
$f\in\cL^\infty(\dom)$. Then 
\begin{equation*}
\df{\Dens_\C}{f} = 
\big[\, \iC f, \sC f \,\big] \,.
\end{equation*}
\end{proposition}

\begin{proof}   
The mapping $\df{\cdot}{f}:\bawl{\dom}\to\R$ is linear. 
By $f\in\bawl{\dom}^*$ it is continuous and, thus, also weak$^*$ continuous. 
Hence it maps the convex weak$^*$ compact set $\df{\Dens_\C}{f}$ onto a convex
compact set in $\R$. Now Theorem~\ref{dm-s5} implies the assertion.
\end{proof}

A closed convex set $M\subset\R^m$ is uniquely determined by its 
{\it support function}
\begin{equation} \label{dm-supp}
  \omega_M(v) := \sup_{w\in M} \df{v}{w}  \qmq{for all} v\in \R^m\,
\end{equation}
(cf. \cite[p. 28]{clarke}).
For some fixed $F\in\cL^\infty(\dom,\R^m)$ ($=\cL^\infty(\dom)^m$)
we can now easily compute the support function of the closed convex set 
$\df{\Dens_\C}{F}\subset\R^m$.

\begin{corollary} \label{dm-s8}
Let $\dom\in \bor{\R^n}$, let $\C\subset\cl{\dom}$ be a density set, 
let $\Dens_\C\ne\emptyset$, and assume that
$F\in\cL^\infty(\dom,\R^m)$. Then the support function $\omega$ of 
$\df{\Dens_\C}{F}$ is given by
\begin{equation*}
  \omega(v) :=  \sC\, (F\cdot v) 
  \qmq{for all} v\in\R^m \,.
\end{equation*}
\end{corollary}

\begin{proof}
By definition of the support function of $\df{\Dens_\C}{F}$ we have
\begin{equation*}
  \omega(v) = \sup_{w\in\df{\Dens_\C}{F}} w\cdot v 
  = \sup_{\me\in\Dens_\C}\sI{\ccl\C}{F\cdot v}{\me} 
  \overset{\reff{dm-s5-1}}{=} \sC\, (F\cdot v) \, 
\end{equation*}
for all $v\in\R^m$.
\end{proof}

\section{Extreme points}
\label{ep}

Let $M$ be a set in a linear space. Then $u\in M$ is an 
{\it extreme point} of $M$ if for every $u_1,u_2\in M$ with $u_1\ne u_2$ 
and for all $\alpha\in[0,1]$ one has that
\begin{equation*}
  u=\alpha u_1+(1-\alpha)u_2 \qmq{implies} \alpha\in\{0,1\}\,.
\end{equation*}
The importance of extreme points follows from the Krein-Milman theorem 
(cf. \cite[p.~154]{eidelman}, \cite[p. 157]{zeidler-iii}).
It says that a compact convex set in a Hausdorff locally convex vector space
is the closure of the convex hull of its extreme points. 
Since $\Dens_\C$ is convex and weak$^*$ compact, the
theorem can be applied with the weak$^*$ closure. Thus we can consider 
$\Dens_\C$ to be spanned by its extreme points. 
Theorem~\ref{ep-s2} below gives a necessary and sufficient condition for a
density measure to be an extreme point. 

A nontrivial measure $\mu\in\bawl{\dom}$ with $\dom\in\cB(\R^n)$ is called 
{\it \zo measure} if 
\begin{equation} \label{ep-e1}
  \me(\bals) = 0 \zmz{or} \me(\bals)=1 
  \qmq{for every} \bals \in \bor{\dom} \,.
\end{equation}
These measures play an special role in $\bawl{\dom}$
(cf. Toland \cite[p. 64]{toland}).
It turns out that \zo measures localize similar as Dirac measures. In
particular they are density measures on $\dom$ at a single point. 

\begin{theorem}\label{ep-s3}  
Let $\dom \in \bor{\R^n}$ and let $\me\in\bawl{\dom}$ be a \zo measure. 
Then 
\begin{equation*}
  \cor\me=\{x\} \qmq{for some} x\in\ccl\dom\,.
\end{equation*}
Moreover $\mu$ is pure and, if $x\in\cl\dom$, we have $\me\in\Dens_x$.
\end{theorem}
\noi
For $x\in\cl\dom$ we set 
\begin{equation*}
  \ZO_x := \big\{\me\in\Dens_x \:\big|\: \mu \text{ is \zo measure} \big\}\,. 
\end{equation*}

\begin{proof}   
Since $\mu\ne 0$, we have by \reff{pl-e0} that $\cor\me\ne\emptyset$ 
and $\cor\me\subset\ccl\dom$.  
Assume that there are $x,y\in\cor{\me}$ with $x\ne y$. We choose  
$\delta>0$ such that $B_{\delta}(x)\cap B_{\delta}(y)=\emptyset$.
Then, without loss of generality,
\begin{equation*}
  \me(B_{\delta}(x)\cap\dom) = 1 \qmq{and} 
  \me(B_{\delta}(y)\cap\dom) \le \me\big(\dom\setminus 
  B_{\delta}(x)\big) = 0 \,, 
\end{equation*}
in contradiction to $y\in\cor{\me}$. Hence $\cor\mu$ is a singleton $\{x\}$ in
$\ccl\dom$ and $\me$ is pure by \reff{pl-e2}. 
If $x\in\cl\dom$, the definition of core gives 
\begin{equation*}
  \me(B_\delta(x)\cap\dom) = \me(\dom) = 1 \qmq{for all} \delta>0\,.
\end{equation*}
Since $\mu\ge 0$, we get $\me\in\Dens_x$.
\end{proof}

Let us characterize the extreme points of $\Dens_\C$ (cf. 
\cite{toland} for that of the unit sphere in $\bawl{\dom}$).

\begin{theorem} \label{ep-s2} 
Let $\dom \in \bor{\R^n}$, let $\C$ 
satisfy \reff{dm-e0a}, and let $\me \in \Dens_\C$.
Then $\me$ is an extreme point of $\Dens_\C$ if and only if
$\mu$ is a \zo measure.  
\end{theorem}

\begin{proof} 
First let $\me\in\Dens_\C$ be a \zo measure. We assume that there are 
$\me_1,\me_2 \in \Dens_\C$ such that
\begin{equation}\label{ep-s2-5}
  \me = \alpha\me_1 + (1-\alpha)\me_2 \qmq{where} 
  \me_1\ne\me_2 \zmz{and} \alpha\in(0,1)\,.
\end{equation}
Then there is some $\bals \in \bor{\dom}$ with
\begin{equation} \label{ep-s2-6}
\me_1(\bals)\neq \me_2(\bals) \,.
\end{equation}
If $\me(\bals)=0$, then $\me_1(\bals) = \me_2(\bals) = 0$ by 
$\me_j\ge 0$, which contradicts \reff{ep-s2-6}.
Hence $\me(\bals) = 1$ and $\me(\dom\setminus B)=0$. Therefore
\begin{equation*}
	\me_1(\dom\setminus B) = \me_2(\dom\setminus B) = 0 
\end{equation*}
by \reff{ep-s2-5} and, thus, $\me_1(\bals) = \me_2(\bals) = 1$. 
This also contradicts \reff{ep-s2-6}. Hence \reff{ep-s2-5} must be wrong
and $\me$ is extreme point of $\Dens_\C$. 

Now let $\me$ be an extreme point of $\Dens_\C$. We assume that there is 
some $\bals \in \bor{\dom}$ with $\me(B)\in(0,1)$. By $\mu(\dom)=1$,
also $\me(\dom\setminus B)\in(0,1)$. We identify the measures 
\begin{equation*}
  \me_1 := \frac{1}{\me(\bals)}\reme{\me}{\bals}  \,, \quad
  \me_2 := \frac{1}{\me(\dom\setminus \bals)}\reme{\me}{(\dom\setminus\bals)} 
\end{equation*}
with their extension onto $\dom$ by zero. Clearly, $\me_1\ne\me_2$ and
$0\le\me_j\le 1$. Moreover
\begin{equation} \label{ep-s2-8}
  \me = \me(\bals) \me_1 + \me(\dom\setminus\bals) \me_2 =
  \me(\bals) \me_1 + \big(1-\mu(B)\big) \me_2 \,.
\end{equation}
Hence, for all $\delta>0$,
\begin{equation*}
  1 = \mu(\C_\delta\cap\dom) = 
  \me(\bals) \me_1(\C_\delta\cap\dom) + 
  \big(1-\mu(B)\big) \me_2(\C_\delta\cap\dom)\,.
\end{equation*}
By $\mu_j\le 1$ we get $\mu_j(\C_\delta\cap\dom)=1$ and, thus, 
$\mu_j\in\Dens_\C$. 
But this combined with \reff{ep-s2-8} is a 
contradiction, since $\me$ is extreme point of $\Dens_\C$. 
Therefore $\mu(B)\in\{0,1\}$ for all $B$ and the assertion follows. 
\end{proof}

\noi
For a Borel set $B$ whose boundary doesn't meet the core of \zo measure
$\me$, we easily get $\me(B)$ from Theorem~\ref{ep-s3} and the definition
of core.  

\begin{corollary} \label{ep-s4} 
Let $\dom\in\bor{\R^n}$, let $\me\in\ZO_x$ 
for some $x \in \cl{\dom}$, and let $B \in \bor{\dom}$. Then we have
\begin{equation*}
		\me(B) = \bigg\{
\begin{array}{ll} 1 &\mbox{ if } x\in \Int{B}\,, \\
			0 &\mbox{ if } x\in \Ext B\,.

\end{array}
\end{equation*}
\end{corollary}

Theorem~\ref{ep-s3} tells us that a \zo measure $\me$ lives on any Ball
around its core $\{x\}$.  
A more precise analysis shows that it concentrates near $x$
in direction of a ray emanating from~$x$.  
For $x\in\R^n$ we thus define the {\it open rotational cone with vertex at} 
$x\in\R^n$, {\it rotational axis} $v\in\R^n\setminus\{0\}$,    
and {\it opening angle} $2\alpha\in(0,\pi)$ by  
\begin{equation}\label{ep-e4}
  K(x,v,\alpha) := \big\{ y\in\R^n \:\big|\:
  y\ne x, \varangle (y-x,v) < \alpha \big\} \,  
\end{equation}
($\varangle$ denotes the angle between vectors).

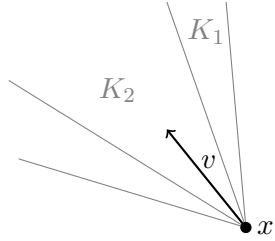
\begin{figure}[H]
	\centering
	\begin{tikzpicture}[scale=1.3]
		\draw (1.2,-1) node {$x$};
		\draw [thick,->] (1,-1) -- node[above] {$\,\vvec$} (0.2,0);
 		\draw[gray] (1,-1) to (0.8,1.3);
		\draw[gray] (1,-1) to (-1.3,-0.3);
		\draw[gray] (0.6,1.0) node {$K_1$};
		\draw[gray] (1,-1) to (0.2,1.3);
		\draw[gray] (1,-1) to (-1.4,0.5);
		\draw[gray] (-0.3,0.4) node {$K_2$};
                \draw[line width=2pt] (1,-1) circle [radius = 0.03];  
	\end{tikzpicture}
	\caption{Cones $K_j=K(x,v,\alpha_j)$ for a large angle $\alpha_1$ and
          a smaller angle $\alpha_2$.}\label{fig:ex_dens_cone}
\end{figure}

\begin{theorem} \label{ep-s5} 
Let $\dom \in \bor{\R^n}$ and let $\me\in\ZO_x$ 
for some $x\in\cl\dom$. 
Then there is a unique $v\in \R^n$ with $|v|=1$ such that 
\begin{equation} \label{ep-s5-1}
\me\big( K(x,v,\alpha)\cap \dom\big) = 1 \qmq{for all}
\alpha \in \big(0,\tfrac{\pi}{2}\big) \,.
\end{equation}
\end{theorem}

\noi
Since $\cor{\me}=\{x\}$, we have that all sets
\begin{equation*}
  K(x,v,\alpha)\cap\dom\cap B_\delta(x) \zmz{with}
  \alpha \in \big(0,\tfrac{\pi}{2}\big)\,, \z \delta>0
\end{equation*}
are aura of $\me$. This readily implies
that \zo measures with different directions $v$ in \reff{ep-s5-1} cannot
agree. A further description of \zo measures by ultrafilters
can be found in \cite{toland}. 

\begin{proof} 
Let us fix $\me\in\ZO_x$ for some $x\in\ol\dom$   
and set
$S^n := \bd B_1(0)$. For contradiction we assume that, 
for all $v\in S^n$,
there is some $\alpha_v\in\big(0,\frac{\pi}{2}\big)$ such that
\begin{equation*}
\me\big(K(x,v,\alpha_v)\cap \dom\big) = 0 \,.
\end{equation*}
Since $x+S^n$ is compact and since it is covered by the 
$K(x,v,\alpha_v)$, there is a finite cover 
\begin{equation*}
  K(x,v_k,\alpha_{v_k}) \qmq{for} k=1,\dots,l\,. 
\end{equation*}
Certainly,
\begin{equation*}
  \R^n\setminus\{x\} \subset \bigcup_{k=1}^l K(x,v_k,\alpha_{v_k}) \,.
\end{equation*}
Hence,
\begin{equation*}
  \me(\dom) = \me\big(\dom\setminus\{x\}\big) \le
  \sum_{k=1}^l \me\big(K(x,v_k,\alpha_{v_k}) \cap \dom\big) = 0 \,,
\end{equation*}
which contradicts $\me(\dom)=1$. Consequently, there is some $v\in S^n$ 
satisfying \reff{ep-s5-1}. Since $\mu$ cannot be $1$ on two disjoint cones, $v$
has to be unique. 
\end{proof}

\section{Precise representative}
\label{app-pr}

Let $\dom\subset\bor{\R^n}$. 
Then we have for a function $f\in L^1_{\rm loc}(\dom)$ that
\begin{equation} \label{app-e0}
  \lim_{\de\downarrow 0} \:\mI{B_\de(x)\cap\dom}{|f-f(x)|}{\lem} = 0 \,
\end{equation}
for $\lem$-a.e. $x\in\cl\dom$ 
(cf. \cite[p. 44]{evans} with $\mu=\reme{\lem}{\dom}$).
We call $x\in\dom$ {\it Lebesgue point} of $f$ if \reff{app-e0} is satisfied.
Notice that Lebesgue points $x$ belong to $\cdom$. For a class 
$f\in\cL^1_{\rm loc}(\dom)$ we thus have for $\lem$-a.e. $x\in\cl\dom$ that
\begin{equation}\label{app-e0a}
  \lim_{\de\downarrow 0} \:\mI{B_\de(x)\cap\dom}{|f-\alpha|}{\lem} = 0 
\end{equation}
for some $\alpha\in\R$. Sometimes $\alpha$ satisfying 
\reff{app-e0a} is said to be the approximate limit of $f$ at $x$
(cf. \cite[p.~160]{ambrosio}). However we want to use a weaker version
(cf. \cite[p.~46]{evans}). 
We call $\alpha\in\R$ {\it approximate limit} of $f\in\cL^1_{\rm loc}(\dom)$ 
at $x\in\cdom$ (with respect to~$\dom$) if
\begin{equation*} \label{dm-app}
  \lim_{\delta\downarrow 0} 
  \frac{\lem\big(B_\delta^\dom(x)\cap\{|f-\alpha|>\eps\}\big)}
       {\lem(B_\delta^\dom(x))}
  = 0 \qmq{for all} \eps >0\,
\end{equation*} 
and we write
\begin{equation*}
  \alim_{y\to x} f(y)=\alpha\,.
\end{equation*}
Let us start with some integrability conditions for integrals related to
density measures. 

\begin{theorem} \label{app-s1}
Let $\dom\in\bor{\R^n}$, let $f\in\cL^1_{\rm loc}(\dom)$, 
let $x\in\cdom$, 
let $\dens_x^\dom$ be an $\lem$-density in $\LDens_x$, and let $\alpha\in\R$. 
\bgl
\item We have that 
  \begin{equation*}
    \reff{app-e0a} \qmq{implies} \alim_{y\to x} f(y) = \alpha
  \end{equation*}
and the converse is true if $f\in\cL^\infty_{\rm loc}(\dom)$.

\item
If we have that 
\begin{equation}\label{app-s1-2}
  \limsup_{\de\downarrow 0} \:\mI{B_\de(x)\cap\dom}{|f|}{\lem} < \infty\,, 
\end{equation}
then $f$ is $dens_x^\dom$-integrable.

\item 
If $\alim_{y\to x} f(y) = \alpha$, 
then $f=\alpha$ i.m. $\dens_x^\dom$ and $f$ is $\dens_x^\dom$-integrable with
\begin{equation*}
  \sI{x}{f}{\dens_x^\dom} = \alpha\,.
\end{equation*}

\item
If \reff{app-e0a} is satisfied, 
then $f$ is $dens_x^\dom$-integrable with
\begin{equation*}\label{app-s1-3}
  \alpha = \sI{x}{f}{\dens_x^\dom} =
  \lim_{\de\downarrow 0} \:\mI{B_\de(x)\cap\dom}{f}{\lem} =
  \alim_{y\to x} f(y)  \,.
\end{equation*}
\el
\end{theorem}
\noi
In \cite[p. 160-162]{ambrosio} the first assertion can be found 
for open $\dom$ and $x\in\dom$. There is also an example given that shows 
that the converse statement may fail for unbounded $f$. Assertions 
(2)-(4) can be found in \cite{trace} where the integrals in (3) and (4) 
are independent of the
special choice of the $\lem$-density $\dens_x^\dom$.  

\begin{proof}
For (1) we can argue analogously to the proof of Proposition~3.65 in
\cite{ambrosio}. (2)-(4) follow from 
Proposition~2.15 in \cite{trace}.
\end{proof}
\noi
Notice that the existence of the approximate limit 
$\alim_{y\to x}f$ does not imply \reff{app-s1-2}
(cf. Example~2.18 in \cite{trace}). Otherwise, the next example shows that
\reff{app-s1-2} does not imply the existence of $\alim_{y\to x}f$.

\begin{example}\label{app-s1b}
Let $\dom=B_1(0)\subset\R^2$. We use polar coordinates $(r,\beta)$ (radius
$r>0$, angle $\beta\in[0,2\pi)$) and define some $f\in\cL^1(\dom)$ by
\begin{equation*}
  f(r,\beta) = \tfrac{1}{\sqrt\beta} \,.
\end{equation*}
Then
\begin{equation*}
  \mI{B_\delta(0)}{f}{\cL^2} = \tfrac{1}{\pi\delta^2}
  \int_0^\delta\int_0^{2\pi}\tfrac{r}{\sqrt\beta}  \,d\beta\,dr =
  \tfrac{2\sqrt{2\pi}}{\pi\delta^2}\int_0^\delta r\,dr =
  \sqrt{\tfrac{2}{\pi}}
\end{equation*}
and, thus, \reff{app-s1-2} is satisfied. However, we readily get for any
$\alpha\in\R$ and $\eps>0$ that 
\begin{equation*}
\lim_{\delta\downarrow 0}
  \frac{\lem\big(B_\delta(0)\cap\{|f-\alpha|>\eps\}\big)}
       {\lem(B_\delta(0))} \ge
  \lim_{\delta\downarrow 0}
  \frac{\lem\big(B_\delta(0)\cap\{f>\alpha+\eps\}\big)}
       {\lem(B_\delta(0))} > 0\,.
\end{equation*}
Therefore, $\alim_{y\to x}f$ does not exist.
\end{example}

By \reff{app-e0}, we have for $f\in L^1_{\rm loc}(\dom)$ 
that \reff{app-e0a} is satisfied  
with $\alpha=f(x)$ for $\lem$-a.e. $x\in\cl\dom$.
Hence,
\begin{equation*} \label{app-e3}
  f(x) = 
  \sI{x}{f}{\dens_x^\dom} =
  \lim_{\de\downarrow 0} \:\mI{B_\de(x)\cap\dom}{f}{\lem}
  \quad \zmz{$\lem$-a.e. on} \cl\dom \,
\end{equation*}
for any $\dens_x^\dom\in\LDens_x$.
This motivates the definition of a {\it precise representative} for 
$f\in\cL^1_{\rm loc}(\dom)$ on $\cl\dom$ either as usual as limit of averages by
\begin{equation}\label{app-s1-5}
  f^\star(x) := \Bigg\{
  \begin{array}{ll}
    \lim\limits_{\de\downarrow 0} \:\mI{B_\de(x)\cap\dom}{f}{\lem} & 
    \zmz{if the limit exists,}  \\[2pt]
    0 & \zmz{otherwise} 
  \end{array}
\end{equation}
(cf. \cite[p. 46]{evans}) or as integral
representation with some  
$\dens_x^\dom\in\LDens_x$ by
\begin{equation}\label{app-s1-6}
  \pr{f}(x) := \bigg\{
  \begin{array}{ll}
    \sI{x}{f}{\dens_x^\dom} & 
    \zmz{if the integral exists,}  \\[2pt]
    0 & \zmz{otherwise.} 
  \end{array}
\end{equation}
Notice that the precise representative is usually only defined for open $\dom$
and $x\in\dom$. We readily see 
that the existence of the limit in \reff{app-s1-5} and of the integral
in \reff{app-s1-6} requires $x\in\cdom$. Theorem~\ref{app-s1} implies
that
\begin{equation}\label{app-e4}
  f^\star(x)=\pr{f}(x) \qmq{if \reff{app-e0a} is satisfied for $x$
  with some $\alpha$} 
\end{equation}
and, for any $\dens_x^\dom\in\LDens_x$,
\begin{equation} \label{app-e4a}
  \alim_{y\to x} f(y)=\alpha  \qmq{implies}
  \pr f(x)=\sI{x}{f}{\dens_x^\dom}=\alpha  \,,
\end{equation}
i.e. $\pr f(x)$ is independent of $\dens_x^\dom$ if the approximate limit at
$x$ exists.

Thus $f^\star$ and $\pr{f}$ can differ at most at points $x\in\dom$ where 
\reff{app-e0a} fails, 
which is an $\lem$-null set. At these points $x$ it is possible that
both the limit in  
\reff{app-s1-5} and the integral in \reff{app-s1-6} exist and are 
equal, that both exist but differ, that only one of them exists, or that 
neither exist (cf. \cite[Remark 2.16]{trace} for examples).
In particular, different choices of $\dens_x^\dom$ in \reff{app-s1-6} can
influence the value $\pr{f}(x)$, but at most for $x$ where $\alim_{y\to x} f(y)$
does not exist. If $f\in \cL^\infty(\dom)$, then 
the integral in \reff{app-s1-6} exists for all $x\in\cdom$ 
while it might depend on the choice of 
$\dens_x^\dom$ for $x$ of an $\lem$-null
set. If $f$ is continuous at $x$, 
then clearly $\pr f(x)=f^\star(x)=f(x)$. The advantage of $\pr f$ is  that an
integral calculus is available, which can simplify various arguments. 

For a refined analysis we define
the {\it upper} and {\it lower  approximate limit} of 
$f\in\cL^1_{\rm loc}(\dom)$ at $x\in\cdom$ by
\begin{eqnarray*}
  f^s(x) \::=\: \alimsup_{y\to x} f(y) 
&:=&  
  \inf\bigg\{ \alpha'\in\R\:\Big|\: \lim_{\delta\downarrow 0}
  \frac{\lem\big(B_\delta^\dom(x)\cap\{f>\alpha'\}\big)}
       {\lem(B_\delta^\dom(x))} = 0 \,\bigg\} \,, \\
  f^i(x)\: :=\:\aliminf_{y\to x} f(y)\, 
&:=&  
  \sup\bigg\{ \alpha'\in\R\:\Big|\: \lim_{\delta\downarrow 0}
  \frac{\lem\big(B_\delta^\dom(x)\cap\{f<\alpha'\}\big)}
       {\lem(B_\delta^\dom(x))} = 0 \,\bigg\} \,.
\end{eqnarray*}
Both limits can attain $\pm\infty$ and, clearly,
\begin{equation*}
  -\infty \le f^i(x) \le f^s(x) \le \infty \qmq{for all} x\in\cdom\,.
\end{equation*}
For $\alpha\in\R$ we readily obtain that
\begin{equation}\label{app-e8}
  \alim_{y\to x} f(y)=\alpha \qmq{if and only if}
  f^i(x)=f^s(x)=\alpha \,.
\end{equation}
The next statement gives a further integrability condition. 

\begin{proposition} \label{app-s1a}
Let $\dom\subset\cB(\R^n)$, let $x\in\cdom$, 
let $\dens_x^\dom\in\LDens_x$, and let 
$f\in\cL^1_{\rm loc}(\dom)$. If $f^i(x)$ and $f^s(x)$ are finite, then
\begin{equation} \label{app-s1a-1}
  f^i(x)\le f \le f^s(x) \zmz{\rm i.m. $\dens_x^\dom$\,.} \quad
\end{equation}
Moreover, $f$ is $\dens_x^\dom$ integrable with
\begin{equation} \label{app-s1a-2}
  f^i(x)\le \sI{x}{f}{\dens_x^\dom} \le f^s(x) \,.
\end{equation}
\end{proposition}
\noi
For $f^i(x)=f^s(x)$ we readily recover (2) from Theorem~\ref{app-s1}.

\begin{proof}
The definition of $\alpha=f^s(x)$ implies that 
\begin{equation*}
  \lim_{\delta\downarrow 0}
  \frac{\lem\big(B_\delta^\dom(x)\cap\{f-\alpha>\eps\}\big)}
       {\lem(B_\delta^\dom(x))} = 0 \qmq{for all} \eps>0\,.
\end{equation*}
Using $\dom_\eps=\{f-\alpha>\eps\}$, we thus get
\begin{equation*}
  0\le \dens_x^\dom(\dom_\eps) \le \limsup_{\delta\downarrow 0}
  \frac{\lem\big(B_\delta^\dom(x)\cap\dom_\eps\big)}{\lem(B_\delta^\dom(x))}
  =0  \qmq{for all} \eps>0 \,.
\end{equation*}
Hence, by definition, $f \le f^s(x)$ i.m. $\dens_x^\dom$. Analogously we can
argue for $f^i(x)$ to get the other inequality in \reff{app-s1a-1}.

With $\beta=1+\max\{|f^i(x)|,|f^s(x)|\}$, we now define simple functions on
$\dom$ by  
\begin{equation*}
  f_k(y) = \bigg\{ 
  \begin{array}{ll} \tfrac{l}{2^k} & \text{on } 
   f^{-1}\big(\big[\tfrac{l}{2^k},\tfrac{l+1}{2^k}\big)\cap[-\beta,\beta]\big) 
       \tx{ for all } l\in\Z \,,\\[1mm]
         0 & \text{otherwise }. 
    \end{array}
\end{equation*}
The $f_k$ approximate $f$ uniformly on $\{|f|\le\beta\}$. By the definition of
$f^i(x)$, $f^s(x)$, we have for any $\eps>0$ that
\begin{eqnarray*}
  0
&\le&
  \dens_x^\dom\big(\{|f_k-f|>\eps\}\cap\{|f|>\beta\}\big)
\:\le\:
  \dens_x^\dom (\{|f|>\beta\}) \\
&\le&
  \limsup_{\delta\downarrow 0}
  \frac{\lem\big(B_\delta^\dom(x)\cap\{|f|>\beta\}\big)}{\lem(B_\delta^\dom(x))}
\:=\:
  0\,.
\end{eqnarray*}
Hence, for all $\eps>0$, 
\begin{equation*}
  \lim_{k\to\infty}\dens_x^\dom(\{|f_k-f|>\eps\})=0
\end{equation*}
and, therefore, $f_k\to f$ i.m. $\dens_x^\dom$. Thus $f$ is
$\dens_x^\dom$-measurable. As above we obtain that $|f|\le\beta$
i.m. $\dens_x^\dom$. Since the constant function $\beta$ is $\mu$-integrable
and majorizes $|f|$, we get that $f$ is $\dens_x^\dom$-integrable
(cf. \cite[p. 113]{rao}). Now \reff{app-s1a-2} follows from \reff{app-s1a-1}
by a standard calculus rule (cf. \cite[p. 106]{rao}). 
\end{proof}

Now we want to apply the previous results to functions of bounded variation.
We assume that $\dom\subset\R^n$ is open and that 
$f\in\cB\cV(\dom)$ is a function of bounded variation. First we collect
several properties from the literature. 
We call $x\in\dom$ a {\it jump point} of $f$ if
\begin{equation*}
  f^i(x) <  f^s(x) \,.
\end{equation*}
Let $\dom_J\subset\dom$ be the set of all jump points and 
$\dom_\infty\subset\dom$ be the set where $f^i$ or $f^s$ is infinite, i.e.
\begin{equation*}
  -\infty < f^i(x)\le f^s(x)<\infty \qmq{on} \dom\setminus\dom_\infty\,.
\end{equation*}
From \cite[p. 210-211]{evans} we obtain
\begin{equation*}
  \lem(\dom_J)=0\,, \qquad \hm(\dom_\infty)=0 \,
\end{equation*}
and, by \cite[p. 168]{ambrosio},
\begin{equation*}
  \tx{\reff{app-s1-2} is satisfied for $\cH^{n-1}$-a.e. $x\in\dom$ }.
\end{equation*}
Moreover, with the notation
\begin{equation*}
  \tilde f(x) := \frac{f^i(x)+f^s(x)}{2}\,,
\end{equation*}
we have
\begin{equation}\label{app-e5}
  \lim_{\de\downarrow 0} \:\mI{B_\de(x)}{|f-\tilde f(x)|}{\lem} = 0
  \qmq{for $\cH^{n-1}$-a.e.} x\in\dom\setminus\dom_J 
\end{equation}
and, for $\cH^{n-1}$-a.e. $x\in\dom_J$, there is some unit vector 
$\nu\in\R^n$ such that 
\begin{equation}\label{app-e6}
  \lim_{\de\downarrow 0} \:\mI{B_\de(x)\cap H_x^-}{|f-f^i(x)|}{\lem} = 0 \,, \quad
  \lim_{\de\downarrow 0} \:\mI{B_\de(x)\cap H_x^+}{|f-f^s(x)|}{\lem} = 0
\end{equation}
where
\begin{equation*}
  H_x^\pm = \{ y\in\R^n\mid \pm \nu\cdot(y-x)>0\}
\end{equation*}
(cf. \cite[p. 213]{evans}, \cite[p. 274]{ziemer}, \cite[p. 171]{ambrosio}). 
By Proposition~\ref{app-s1a} we have that 
\begin{equation*}
  \tx{$f$ is $\dens_x^\dom$-integrable and satisfies \reff{app-s1a-2}}
  \qmq{if} \dens_x^\dom\in\LDens_x\,,\;
  x\in\dom\setminus\dom_\infty\,.
\end{equation*}
Consequently,
\begin{equation*}
  \pr f(x) \qmq{exists for all}  x\in\dom\setminus\dom_\infty\,.
\end{equation*}
Theorem~\ref{app-s1} and Proposition~\ref{app-s1a} now imply the next
representations (that are partially known from e.g. 
\cite[p. 213, 216]{evans}, \cite[p. 276]{ziemer}). 
We use the notation 
$\LDens_x^\pm$ for the set of $\lem$-densities at $x$ within
$\dom\cap H_x^\pm$ (instead of $\dom$) and
$\alim_{y\to x,\: y\in H_x^\pm}$ denotes the approximate limit taken
in $\dom\cap H_x^\pm$.

\medskip

\begin{theorem} \label{app-s2}
Let $\dom\subset\R^n$ be open and let $f\in\cB\cV(\dom)$.
\bgl
\item
We have 
that 
\begin{equation}\label{app-s2-1}
  \tilde f(x) = \pr{f}(x) = f^\star(x) 
  \qquad\text{for\z $\cH^{n-1}$-a.e.} \z x\in\dom\setminus\dom_\infty\,.
\end{equation}
Moreover,
\begin{equation}\label{app-s2-3}
  \tilde f(x) = \pr{f}(x) = \alim_{y\to x} f(y) 
  \qquad \text{for all} \z x\in\dom \setminus (\dom_J\cup\dom_\infty)\,,
\end{equation}
\begin{equation} \label{app-s2-2}
  \tilde f(x) = f^\star(x) = \alim_{y\to x} f(y) 
  \qquad \text{for\z $\cH^{n-1}$-a.e.} \z 
  x\in\dom \setminus(\dom_J\cup\dom_\infty)\,.
\end{equation}
In particular, $\pr f(x)$ is independent of $\dens_x^\dom\in\LDens_x$ in 
\reff{app-s2-1} and \reff{app-s2-3} and $\alim_{y\to x} f(y)$ does not
exist on $\dom_J\cup\dom_\infty$.
 
\item
We have for any $\dens_x^\pm\in\LDens_x^\pm$ and  
$\cH^{n-1}$-a.e. $x\in\dom_J$ that
\begin{eqnarray}
  f^i(x) 
&=& \label{app-s2-5}
  \sI{x}{f}{\dens_x^-} \:=\: 
  \lim_{\de\downarrow 0} \:\mI{B_\de(x)\cap H_x^-}{f}{\lem} 
  \:=\: \alim_{y\to x,\: y\in H_x^-} f(y) \,,\quad  \\
f^s(x) 
&=& \label{app-s2-6}
 \sI{x}{f}{\dens_x^+} \:=\: 
  \lim_{\de\downarrow 0} \:\mI{B_\de(x)\cap H_x^+}{f}{\lem} 
  \:=\: \alim_{y\to x,\: y\in H_x^+} f(y) \,.\quad
\end{eqnarray}
At least one of the approximate limits in \reff{app-s2-5} and \reff{app-s2-6}
does not exist on $\dom_\infty$. 
\el
\end{theorem}
\noi 
Let us mention that Theorem~\ref{app-s1} readily implies sharper results
for Sobolev functions $f\in\cW^{1,p}(\dom)$. If $1< p<n$, then
\reff{app-s2-1}  
is satisfied for all 
$x\in\dom$ up to a set of $p$-capacity zero 
(use \cite[p. 147, 160]{evans}). In the case $p>n$, $f$ has a continuous
representative (cf. \cite[p. 143]{evans}) and, thus, 
\reff{app-s2-1} is satisfied for all $x\in\dom$ and $\pr f$ coincides with
the continuous representative. 

\begin{proof}
For $x\in\dom\setminus(\dom_J\cup\dom_\infty)$ we have that
$\alim_{y\to x} f(y)=\tilde f(x)$ by \reff{app-e8}. 
Then \reff{app-s2-3} follows from Theorem~\ref{app-s1}~(3) where 
$\pr f(x)$ is independent of $\dens_x^\dom\in\LDens_x$. \reff{app-e5} and
Theorem~\ref{app-s1}~(4) imply \reff{app-s2-2}. The same way, 
with $\dom\cap H^\pm$ instead of $\dom$, we get 
\reff{app-s2-5} and \reff{app-s2-6} from \reff{app-e6}.
Clearly, $\alim_{y\to x} f(y)$ does not exist on $\dom_\infty$ by \reff{app-e8}
and the approximate limits in \reff{app-s2-5}, \reff{app-s2-6}
do not exist if $f^i(x)$ or $f^s(x)$ is infinite.  

We have \reff{app-s2-1} for 
$\hm$-a.e. $x\in\dom \setminus(\dom_J\cup\dom_\infty)$ by 
\reff{app-s2-3} and \reff{app-s2-2}. For $\hm$-a.e. $x\in\dom_J$
we fix any $\dens_x^\dom\in\LDens_x$. Then 
$\dens_x^\pm:=2\reme{\dens_x^\dom}{H^\pm_x}$ belongs to $\LDens_x^\pm$.
By \reff{app-s2-5}, \reff{app-s2-6} and
$\dens^\dom_x=\frac12(\dens^+_x+\dens^-_x)$, we obtain 
\begin{equation*}
  \tilde f(x) = 
  \tfrac12\Big(\,\sI{x}{f}{\dens_x^+}+ \sI{x}{f}{\dens_x^-}\Big) =  
  \sI{x}{f}{\dens_x^\dom} = \pr f(x)
\end{equation*}
and
\begin{equation*}
  \tilde f(x) = \tfrac12\Big(
\lim_{\de\downarrow 0} \:\mI{B_\de(x)\cap H^+}{f}{\lem} + 
\lim_{\de\downarrow 0} \:\mI{B_\de(x)\cap H^-}{f}{\lem} \Big) =
 \lim_{\de\downarrow 0} \:\mI{B_\de(x)}{f}{\lem} = f^\star(x)\,.
\end{equation*}
Thus we have \reff{app-s2-1} for $\hm$-a.e. 
$x\in\dom \setminus\dom_\infty$.
\end{proof}

\smallskip

Let us now discuss the (boundary) trace of functions $f\in\cB\cV(\dom)$. 
For that we assume that $\dom\subset\R^n$ is open and bounded with 
Lipschitz boundary $\partial\dom$. Then there is a trace function
$f^{\partial\dom}$ on $\partial\dom$ such that
\begin{equation*}
  \lim\limits_{\de\downarrow 0} \:
  \mI{B_\de(x)\cap\dom}{\big|f-f^{\partial\dom}(x)\big|}{\lem} = 0
  \qquad \text{for\z $\cH^{n-1}$-a.e.} \z x\in\partial\dom \,
\end{equation*}
(cf. \cite[p. 181]{evans}).
From Theorem~\ref{app-s1} 
we directly obtain the next representations, whereby the equality 
with $f^\star$ is standard (cf. \cite[p. 181]{evans}).
By $\cW^{1,p}(\dom)\subset\cB\cV(\dom)$ also Sobolev functions are covered.  

\begin{proposition}  \label{app-s3}
Let $\dom\subset\R^n$ be open and bounded with Lipschitz boundary, 
let $\dens_x^\dom$ be an $\lem$-density at $x$ within $\dom$, 
and let $f\in\cB\cV(\dom)$. Then
\begin{eqnarray*}
  f^{\partial\dom}(x) 
&=&
  \pr{f}(x) \:=\: f^\star(x) 
  \qquad \text{for\z $\cH^{n-1}$-a.e.} \z x\in\partial\dom  \,.
\end{eqnarray*}
\end{proposition}

For open bounded $\dom\subset\R^n$ with Lipschitz boundary $\partial\dom$,
the trace function $f^{\partial\dom}$ is an important ingredient of the
Gauss-Green formula. We have e.g. for any Sobolev function  
$f\in\cW^{1,1}(\dom)$ that
\begin{equation} \label{gg-e1}
  \I{\dom}{f\divv\phi}{\lem} + \I{\dom}{\phi\cdot Df}{\lem} =
  \I{\bd\dom}{f^{\partial\dom}\phi\cdot\nu^{\bd\dom}}{\cH^{n-1}}
\end{equation}
for all $\phi\in C^1(\ol\dom,\R^n)$ where $\nu^{\bd\dom}$ denotes the outer
unit normal of $\dom$ (cf. \cite[p. 168]{pfeffer}). The necessity of a trace
function  
$f^{\partial\dom}$ and of an outer unit normal $\nu^{\partial\dom}$ on the
boundary $\bd\dom$ is an essential limitation of this fundamental formula. 
The integral representation of $f^{\partial\dom}$ by $\pr f$
in terms of measures
$\dens_x^\dom$ suggests that we should look for a replacement of the boundary
integral on $\bd\dom$ with the $\si$-measure $\hm$ by an integral on
$\dom$ with a pure measure having core in $\bd\dom$. This way one might
avoid the necessity of a trace function and a normal field defined on 
$\bd\dom$. 
For a suitable replacement of the normals $\nu^{\partial\dom}$, one needs
some vector field $\nu^\dom$ on $\dom$ that
approximates the normals $\nu^{\bd\dom}$ near $\bd\dom$. The gradients of
the Lipschitz continuous distance function $\distf{\bd\dom}$,
that is differentiable $\lem$-a.e., appear to be an appropriate
tool. We define the {\it normal field} of $\dom$ by
\begin{equation*}
  \nu^\dom:= -D\distf{\bd\dom}\,
\end{equation*}
which approximates $\nu^{\bd\dom}$ near the boundary $\bd\dom$.
Moreover, the $\si$-measure $\reme{\cH^{n-1}}{\bd\dom}$, considered as 
functional on $C(\ol\dom)$, can be extended by the Hahn-Banach theorem 
to a linear continuous functional on $\cL^\infty(\dom)$ and, thus, provides 
some pure measure $\om$ in $\bawl{\dom}$ such that
\begin{equation}\label{gg-e2}
  |\om|(\dom)=\hm(\bd\dom)\,, \quad
  \cor{\om}\subset\bd\dom\,,\quad
  \I{\bd\dom}{\tf}{\hm} = \sI{\bd\dom}{\tf}{\om}  \zmz{for all} C(\ol\dom)\,
\end{equation}
and the normalized measure $\frac{\om}{|\om|(\dom)}$ 
is obviously a density measure at $\bd\dom$
(cf. \cite[p. 25]{trace}). However, the extension is not unique, 
while a special one is needed for a replacement in \reff{gg-e1}.
It turns out that there is some $\om\in\bawl{\dom}$ satisfying \reff{gg-e2}
such that, for any $f\in\cW^{1,1}(\dom)$,
\begin{equation} \label{gg-1}
  \I{\dom}{f\divv\phi}{\lem} + \I{\dom}{\phi\cdot Df}{\lem} =
  \sI{\bd\dom}{f\phi\cdot\nu^{\dom}}{\omega} \qmq{for all}
  \phi\in\cW^{1,\infty}(\dom,\R^n)
\end{equation}
(cf. \cite[p. 147]{trace}).
Notice that the integral on the right hand side is in fact an integral 
on~$\dom$. Thus there is no need for $f$, $\tf$, and $\nu^\dom$ to be defined
on $\bd\dom$. Hence we do not need 
\bgl[--]
\item a pointwise trace function $f^{\bd\dom}$ on $\bd\dom$, 
\item test functions $\tf$ defined up to the boundary, and
\item pointwise normals $\nu^{\bd\dom}$ on the boundary.
\el
We can consider $\nu^\dom\omega$ as (vector-valued) normal
measure on $\dom$ concentrated near $\bd\dom$.

This approach allows substantial extensions of the Gauss Green formula. 
For demonstration let us sketch how $\dom$ with inner
boundaries can be treated. Let $\dom\subset\R^2$
consist of two touching squares $\dom_j$ as e.g.  
\begin{equation*}
  \dom=\dom_1\cup\dom_2=\big((0,1)\cup(1,2)\big)\times(0,1)  \,.
\end{equation*}
For some $f\in\cW^{1,1}\dom)$, we first have \reff{gg-1} for each $\dom_j$
separately with pure measures $\om_j$ on $\dom_j$. 
Here $\dom_j$ is an aura of $\om_j$ and, roughly speaking, 
$\om_j$ only sees values of $f$ and $\tf$ on $\dom_j$. Extending the $\om_j$
on $\dom$ by zero and summing up, we get a Gauss-Green formula on all of
$\dom$ with $\om=\om_1+\om_2$ and for test functions
$\tf\in\cW^{1,\infty}(\dom,\R^n)$ that do not have to be extendable up to
$\bd\dom$. This allows a more flexible analysis. If the inner
boundary corresponds e.g. to some crack within some material occupying
$\dom$, we can analyze the behavior of the material on both sides of the
crack separately by choosing test functions that vanish either on $\dom_1$ or
on $\dom_2$.   
For a comprehensive presentation of the treatment of traces by means of
finitely additive measures and related extensions of the Gauss-Green formula, 
we refer to Schuricht \& Schönherr~\cite{trace}.

\section{More explicit examples to some special problems}
\label{app-sm}

For $\dom=(0,1)\subset\R$ it turns out that there are measures 
$\mu\in\op{ba}(\dom,\cB(\dom),\cL^1)$ such that
\begin{equation*}
  \int_0^{1-\frac1k} f\,d\mu = 0 \z\zmz{for all}\z 
  f\in\cL^\infty(\dom)\,,\: k\in\N
  \z\qmq{but}\z \int_0^1 1\,d\mu = 1\,.
\end{equation*}
Toland writes in the preface of his book \cite{toland} 
that ``the target audience is anyone who feels nervous'' about the existence
of uncountably many such measures $\mu$. For a better understanding of this
situation let us provide a more explicit example for $\mu$. It turns out that
\begin{equation*}
  \mu = \dens_0^\dom\,,
\end{equation*}
as constructed in Example~\ref{pl-s2},
is a pure measure with that property. To see this we use that
$\cor{\mu}=\{0\}$, that $\alim_{y\to 0}1=1$,
and we apply Theorem~\ref{app-s1} (3). 

For a further example let $\dom\subset\R^n$ be open. 
In Toland \cite[p. 28]{toland} it is shown with the Hahn-Banach Theorem
that there is some nontrivial $\me\in\bawl{\dom}$ such that
\begin{equation} \label{app-e2}
  \I{\dom}{f}{\me}=0 \quad\qmz{for all continuous} f\in L^\infty(\dom) \,
\end{equation}
(cf. also \cite[p. 57]{yosida}).
Let us provide also here some examples for such $\me$ that are more explicit.
For $x\in\cdom$ we choose disjoint sets $\E_1,\E_2\subset\dom$ with 
\begin{equation}\label{app-sm3}
  \lem(\E_j\cap B_\delta(x))>0 \qmq{for all} \delta>0\,, \z j=1,2\,.
\end{equation}
Then $x\in\cl E_1,\cl E_2$. Applying Corollary~\ref{dm-s2c}
with $\E_j$ instead of $\dom$, we obtain $\lem$-densities 
$\dens_x^{\E_j}$ at $x$ in $\bawl{\E_j}$, that we identify with their zero
extension on $\dom$. This way we get nontrivial density measures 
$\dens_x^{\E_j}\in\bawl{\dom}$ at $x$ with aura $\E_j$. For
\begin{equation*}
  \me := \dens_x^{\E_2}- \dens_x^{\E_1}\in\bawl{\dom}\,
\end{equation*}
we then have
\begin{equation*}
  \me\ne 0\,, \quad \me(\dom)=0 \,.
\end{equation*}
Theorem~\ref{app-s1} (3) with $\E_j$ instead of
$\dom$ implies
\begin{equation*}
  f(x) = \alim_{y\to x} f(y) = \sI{x}{f}{\dens_x^{\E_j}} 
  \qmz{for continuous} f\in L^\infty(\dom)\,, \z j=1,2\,. 
\end{equation*}
Hence, for any continuous $f\in\cL^\infty(\dom)$, 
\begin{equation*}
  \I{\dom}{f}{\me}=\I{\dom}{f}{\dens_x^{E_2}} - \I{\dom}{f}{\dens_x^{E_1}} =
  f(x)-f(x) = 0\,,
\end{equation*}
i.e. \reff{app-e2} is satisfied. Since there is much freedom in selecting the
$E_j$, there is a large variety of nontrivial measures satisfying
\reff{app-e2}.  

More generally, let $\dom\subset\cB(\R^n)$, and let $x\in\cdom$. 
With $E_j$ as in \reff{app-sm3} we can construct some $\me\ne 0$ as above. 
By Proposition~\ref{dm-s4} we then have
\begin{equation*}
  \I{\dom}{f}{\me}=0 
\end{equation*}
for all $f\in\cL^\infty(\dom)$ with $\ix f=\sx f$.

\section{Generalized gradients}
\label{cgg}

Let $X$ be a Banach space and let $f:X\to\R$ be a locally Lipschitz
continuous function. The {\it generalized directional derivative} 
of $f$ at $x\in X$ in direction $v\in X$ (in the sense of Clarke) is 
\begin{equation} \label{cgg-e1}
  f^\circ(x;v) = 
  \limsup_{\substack{y\to x\\t\downarrow 0}} \frac{f(y+tv)-f(y)}{t} \,.
\end{equation}
As function of $v$ it is positively 1-homogenous,
subadditive and, thus, convex and it satisfies $|f^\circ(x;v)|\le c\|v\|$
for all $v\in X$ where $c\ge 0$ is a Lipschitz constant of $f$ near $x$.
The {\it generalized gradient} of $f$ at $x$ is the convex weak$^*$ compact
set 
\begin{equation} \label{cgg-e0}
  \partial f(x) = \{f^*\in X^*\mid \df{f^*}{v}\le f^\circ(x;v)
  \text{ for all } v\in X\} \,
\end{equation}
and $f^\circ(x;\cdot)$ is the support function of $\partial f(x)$
(cf. Clarke \cite{clarke}).

We study the special case that $f:\R^n\to\R$ is locally
Lipschitz continuous. Then we have
$f\in\cW^{1,\infty}_{\rm loc}(\R^n)$ (cf. \cite[p. 131]{evans}).
By Rademacher's theorem, the derivative $Df$ of $f$ exists $\lem$-a.e. (cf.
\cite[p. 81]{evans}). We set
\begin{equation*}
  D_f := \big\{ y\in\R^n\:\big|\: f\mbox{ is differentiable at }y \big\} \,.
\end{equation*}
Then we get the next representation of $f^\circ(x;v)$.

\begin{proposition} \label{cgg-s0}
Let $f:\R^n\to\R$ be locally Lipschitz continuous and let $x\in\R^n$. Then
\begin{equation} \label{cgg-s0-1}
  f^\circ(x;v) = \tesup{x}\: Df\cdot v 
  = \widetilde{\sup_{x}}\: Df\cdot v
  \qmq{for all} v\in\R^n
\end{equation}
with
\begin{equation*}
  \widetilde{\sup_{x}}\: Df\cdot v :=
  \lim_{\delta\downarrow 0}\, \sup_{B_\delta(x)\cap D_f} Df\cdot v =
  \inf_{\delta>0}\, \sup_{B_\delta(x)\cap D_f} Df\cdot v \,.
\end{equation*}
\end{proposition}
\noi
It is worth to notice that $\widetilde{\op{ess\mspace{2mu}sup}}$
agrees with $\widetilde\sup$, i.e. 
for the $\cL^\infty$-function $Df$ it does not matter in \reff{cgg-s0-1} 
if we consider $Df$ only on $D_f\setminus N$ for some $\cL^n$-null set $N$.  
The result is implicitly contained in the Corollary of \cite[p. 64]{clarke}.
For the proof we basically argue as in Clarke \cite[p. 63-64]{clarke}, but
we do not use $\partial f(x)$ and its properties. 

\begin{proof}
Let us fix $x\in\R^n$ and $v\in\R^n$. Then 
there is a ball $B\subset\R^n$ centered at $x$ such that
$f\in\cW^{1,\infty}(B)$ and the definition of the generalized
directional derivative gives
\begin{equation} \label{cgg-s0-5}
   f^\circ(x;v) = \lim_{\delta\to 0}
  \sup_{\substack{y\in B_\delta(x)\\0<t<\delta}} \frac{f(y+tv)-f(y)}{t} 
\end{equation}
(cf.  \cite[p. 38 with $\delta=1$]{clarke}). 
Using lines $L_{y,v}:=\{y+tv\mid t\in\R\}$, we set
\begin{equation*}
  AC_{f,v}:= \big\{ y\in B \:\big|\: 
             f\text{ is absolutely continuous on } L_{y,v}\cap B \big\}\,.
\end{equation*}
Since $f$ is absolutely continuous on \mbox{$\cH^{n-1}$-a.e.} line
directed as $v$, the set $AC_{f,v}$ has full $\cL^n$-measure in $B$
(cf. \cite[p.~164]{evans}). 
For any $\cL^n$-null set $N$, 
\begin{equation*}
  D_0 := \big\{ y\in B \:\big|\:
  \cL^1\big( (L_{y,v}\cap B)\setminus ( D_f\setminus N) \big) > 0 \big\}
\end{equation*}
is an $\lem$-null set. Hence $D:=AC_{f,v}\setminus D_0$ has full
$\lem$-measure in $B$ and, thus, it is dense in $B$. 
Obviously $L_{y,v}\cap B\subset D$ if $y\in D$. Moreover,
$D_f\setminus N$ has full $\cL^1$-measure on all segments 
$L_{y,v}\cap B$ belonging to $D$.
For small $\delta>0$ and $\gamma:=\delta(1+|v|)$ we have by continuity of
$f$ 
\begin{eqnarray*}
  \sup_{\substack{y\in B_\delta(x)\\0<t<\delta}} 
  \frac{f(y+tv)-f(y)}{t}
&=&
  \sup_{\substack{y\in B_\delta(x)\cap D\\0<t<\delta}} 
  \frac{f(y+tv)-f(y)}{t} \\
&=&
  \sup_{\substack{y\in B_\delta(x)\cap D\\0<t<\delta}} 
  \frac{1}{t} \int_0^t Df(y+sv)\cdot v\: ds  \\
&\le&
  \essup{y\in B_{\gamma}(x)\cap D_f}{Df(y)\cdot v} 
\:\le\:
  \sup_{y\in B_{\gamma}(x)\cap D_f} Df(y)\cdot v \,.
\end{eqnarray*}
For each $y\in B_\gamma(x)\cap D_f$ there is $t_{y,\gamma}\in(0,\delta)$ 
such that $y+t_{y,\gamma}v\in B_\gamma(x)$ and
\begin{equation*}
  Df(y)\cdot v \le \frac{f(y+t_{y,\gamma}v)-f(y)}{t_{y,\gamma}} + \delta \,.
\end{equation*}
Taking the limit $\delta\to 0$, we get from \reff{cgg-e1} and \reff{cgg-s0-5}
\begin{eqnarray*}
  f^\circ(x;v) 
&\le&
  \lim_{\delta\to 0}\sup_{y\in B_\gamma(x)\cap D_f} 
  \frac{f(y+t_{y,\gamma}v)-f(y)}{t_{y,\gamma}} + \delta
  \le f^\circ(x;v) \,
\end{eqnarray*}
which implies \reff{cgg-s0-1}.
\end{proof}

Now we formulate two characterizations of the generalized gradient 
$\partial f(x)$.
While the first one is due to Clarke \cite[p.63]{clarke},
the second one seems to be new. Notice that we have 
$Df\in\cL^\infty_{\rm loc}(\dom,\R^n)$.

\begin{theorem} \label{cgg-s1}
Let $f:\R^n\to\R$ be locally Lipschitz continuous, let 
$N\subset\R^n$ be an $\cL^n$-null set, and let $x\in\R^n$. Then
\begin{eqnarray}
  \partial f(x) 
&=& \label{cgg-s1-1}
  \op{conv} \big\{ \lim_{k\to\infty} Df(x_k)\:\big|\: 
  x_k\to x,\; x_k\in D_f\setminus N,\; \lim_{k\to\infty} Df(x_k)
  \text{ exists\,}  \big\}\, \\
&=& \label{cgg-s1-2}
  \df{\Dens_x}{Df} 
\end{eqnarray}
where $\op{conv}$ denotes the convex hull.
\end{theorem}
\noi
 
\begin{proof}
\reff{cgg-e0} implies that $f^\circ(x;v)$ is the support function of
the closed convex set $\partial f(x)$. By Corollary~\ref{dm-s8} 
and Proposition~\ref{cgg-s0}, it is also
support function of $\df{\Dens_x}{Df}$. The uniqueness of the support function
gives \reff{cgg-s1-2} (cf. \cite[p. 29]{clarke}). 

Let $M$ be the set of limits on the right hand side in \reff{cgg-s1-1}. 
For $\mu_x\in\ZO_x$ we have that $Df$ is $\mu_x$-integrable. 
Applying Proposition~\ref{dm-s4} to the components of  
$F=\sI{x}{Df}{\mu_x}$, we get 
\begin{equation*}
   \mu_x\big(\{|Df-F|<\eps\}\big) = 1 \qmq{for all} \eps>0\,.
\end{equation*}
Using $\cor\mu_x=\{x\}$, we obtain that $\{|Df-F|<\eps\}\cap B_\de(x)$ has
positive $\lem$-measure for all $\eps,\de>0$. Hence, we can find 
\begin{equation*}
  x_k\in \{|Df-F|<\tfrac1k\}\cap B_{\frac1k}(x) \qmq{for all} k\in\N\,.
\end{equation*}
Thus $x_k\to x$ and $Df(x_k)\to F$. Consequently, $\df{\ZO_x}{Df}\subset M$.
Taking the convex hull, $\df{\Dens_x}{Df}\subset\op{conv}M$
by Theorem~\ref{ep-s2}. 
Let now $F\in M$. Then $Df(x_k)\to F$ for some sequence $x_k\to x$ and, by
\reff{cgg-s0-1}, 
\begin{equation*}
  Df(x_k)\cdot v \: \to \: F\cdot v \: \le \: \widetilde{\sup_{x}}\:
  Df\cdot v \: = \: f^\circ(x;v) 
  \qmq{for all} v\in\R^n\,.
\end{equation*}
Since $f^\circ(x;v)$ is support function of the convex set
$\partial f(x)=\df{\Dens_x}{Df}$, 
we finally get that $\op{conv}M\subset\df{\Dens_x}{Df}$.
\end{proof}

With the representation from Proposition~\ref{cgg-s1} we can now easily derive 
fundamental calculus rules 
(cf. \cite[p. 38-48]{clarke}).

\begin{proposition} \label{cgg-s3}
Let $f,g:\R^n\to\R$ be locally Lipschitz continuous. 

\bgl
\item
{\rm (Scalar multiple)} For $x\in\R^n$ we have
\begin{equation*}
  \partial(sf)(x) = s\partial f(x)
  \qmq{for all} s\in\R\,.
\end{equation*}
\item
{\rm (Sum rule)} For $x\in\R^n$ we have
\begin{equation*}
  \partial\big(\alpha f(x) +\beta g(x)\big)(x) \:\subset\:
  \alpha \partial f(x) + \beta\partial g(x) 
  \qmq{for all} \alpha,\beta\in\R\,.
\end{equation*}

\item
{\rm (Product rule)} For $x\in\R^n$ we have
\begin{equation*}
  \partial\big(fg\big)(x) \subset f(x)\partial g(x) + g(x)\partial f(x) \,. 
\end{equation*}
\el
\end{proposition}

\begin{proof}
For (1) we use
\begin{equation*}
  \big\langle \Dens_x, D(sf) \big\rangle =
  s\big\langle \Dens_x, Df \big\rangle
\end{equation*}
and, by \reff{cgg-s1-2}, the statement follows. 

For (2) we have
\begin{eqnarray*}
  \big\langle \Dens_x,D\big(\alpha f +\beta g \big) \big\rangle 
&=&
  \big\langle \Dens_x, \alpha Df + \beta Dg \big\rangle \\
&\subset&
  \alpha \langle\Dens_x,Df\rangle + \beta \langle\Dens_x,Df\rangle  \,.
\end{eqnarray*}
and \reff{cgg-s1-2} gives the stated inclusion. 

For (3) we first realize 
that, by continuity,  $f-f(x)=0$ i.m. $\mu$ for any
$\mu\in\Dens_x$ and, hence, 
\begin{equation} \label{cgg-s3-10}
  \sI{x}{fDg}{\mu} = f(x)\sI{x}{Dg}{\mu} + \sI{x}{\big(f-f(x)\big)Dg}{\mu}
  = f(x)\sI{x}{Dg}{\mu}\,.
\end{equation}
Consequently,
\begin{eqnarray*}
  \big\langle \Dens_x,D(fg) \big\rangle 
&=&
  \big\langle \Dens_x,fDg+gDf \big\rangle \\
&\subset&
  f(x) \langle \Dens_x,Dg \rangle + g(x) \langle \Dens_x,Df \rangle 
\end{eqnarray*}
and the assertion follows from \reff{cgg-s1-2}. 
\end{proof}

\bibliographystyle{plain}

\vspace{5mm}

\noi
Moritz Schönherr\\
TU Dresden, Fakultät Mathematik, 01062 Dresden, Germany \\
email: {\it moritz.schoenherr@posteo.de}

\bigskip

\noi
Friedemann Schuricht (corresponding author)\\
TU Dresden, Fakultät Mathematik, 01062 Dresden, Germany\\
email: {\it friedemann.schuricht@tu-dresden.de}

\end{document}